\documentclass[smallextended,final]{svjour3}
\smartqed 
\usepackage{amssymb,latexsym}
\usepackage{amsmath}
\usepackage[labelsep=period]{caption}
\usepackage{tikz}
\usepackage{verbatim}
\usetikzlibrary {positioning}
  \usetikzlibrary{arrows,topaths}
       \usetikzlibrary{decorations,backgrounds}
    \usetikzlibrary{shapes}
\usetikzlibrary{calc}   
\usetikzlibrary{decorations.pathreplacing,angles,quotes}
\usepackage{rawfonts}
\usepackage{psfrag}
\usepackage{amsfonts}
\usepackage{lineno,hyperref}
\definecolor {processblue}{cmyk}{0.96,0,0,0}
\newcommand{\F}{\mathcal F}
\newcommand{\R}{\mathbb R} 

\usepackage{xparse}
\usepackage[margin=0.75in]{geometry}

\tikzstyle{vertex}=[circle,black,draw,fill,text=white,inner sep=2pt]

\NewDocumentCommand{\lockboundingbox}{}{\path[use as bounding box] (current bounding box.south west) rectangle (current bounding box.north east);}

\makeatletter
\NewDocumentCommand{\writecoords}{m}{%
	\node (wcno) at (#1){};
	\path (wcno);
	\pgfgetlastxy{\ASJ@wcx}{\ASJ@wcy}
	\path (wcno) node[anchor=north,fill=white]{\footnotesize(\ASJ@wcx,\ASJ@wcy)};
}
\makeatother

\NewDocumentCommand{\setupcoords}{}{
	\pgfmathsetmacro{\rad}{4}
	\foreach \v [count=\c from 0] in {01456,01356,02356,02357,02578,04578,01458}\coordinate (p\v) at (90+360/7*\c:\rad);

	\foreach \u/\v in {01236/04578,02478/01356,14578/02356,23567/01458}\coordinate (p\u) at ($(0,0)!-2!(p\v)$){};
	\coordinate (p13468) at ($(0,0)!2!(p01456)$);
	\coordinate (p01236) at ($(p01356)!-0.55!(p01458)$);
	\coordinate (p14578) at ($(p01458)+(p01356)-(p01236)$);
	\coordinate (p02478) at ($(p01236)!1!-36:(p14578)$);
	\foreach \u/\v/\w in {23567/02478/14578,13468/01236/23567}\coordinate (p\u) at ($(p\v)!1!108:(p\w)$);

	\foreach \v in {01356,02356,02357,02578,04578,01458,01456,01236,02478,14578,23567,13468}{
		\node[vertex] at (p\v){};
		\node[anchor=south] at (p\v){\small\v};
	}
	\colorlet{col0}[rgb]{red>wheel,0,9}
	\colorlet{col1}[rgb]{red>wheel,1,9}
	\colorlet{col2}[rgb]{red>wheel,2,9}
	\colorlet{col3}[rgb]{red>wheel,3,9}
	\colorlet{col4}[rgb]{red>wheel,4,9}
	\colorlet{col5}[rgb]{red>wheel,5,9}
	\colorlet{col6}[rgb]{red>wheel,6,9}
	\colorlet{col7}[rgb]{red>wheel,7,9}
	\colorlet{col8}[rgb]{red>wheel,8,9}
}
\begin{document}
\title{Petruska's question on planar convex sets}

\author{Adam S. Jobson \and Andr\'e E. K\'ezdy\and Jen\H{o} Lehel
\and\\ Timothy J. Pervenecki \and G\'eza T\'oth\thanks{G. T\'oth was   supported by the National Research, Development and Innovation Office, NKFIH, K-111827}
}

\institute{Adam S. Jobson \at
              University of Louisville,  Louisville, KY 40292\\            
              \email{adam.jobson@louisville.edu}     
  \and
           Andr\'e E. K\'ezdy \at
           University of Louisville,  Louisville, KY 40292\\            
              \email{andre.kezdy@louisville.edu}   
   \and
           Jen\H{o} Lehel \at                                 
              Alfr\'ed R\'enyi Mathematical Institute$^*$,
              Budapest, Hungary\\               
              \email{lehelj@renyi.hu}\\
              and University of Louisville,  Louisville, KY 40292
	\and 
	Timothy J. Pervenecki \at 
	University of Louisville, Louisville, KY 40292\\
	 \email{timothy.pervenecki@louisville.edu} 
	\and 
	G\'eza T\'oth \at
Alfr\'ed R\'enyi Mathematical Institute\footnote{}, Budapest, Hungary\\ 
              \email{geza@renyi.hu}
    }

\date{Received: date / Accepted: date}
\maketitle

\begin{abstract} 
Given $2k-1$ convex sets in $\R^2$ such that no point 
of the plane is covered by more than $k$ of the sets, is it true that there are two among the convex sets whose union contains all $k$-covered points of the plane? This question due to Gy. Petruska  has an obvious affirmative answer for $k=1,2,3$; we show here that the claim is also true for $k=4$,
and we present a counterexample for $k=5$. We explain how
Petruska's  geometry question fits into the classical hypergraph extremal problems, called arrow problems,  proposed by P. Erd\H{o}s.
\end{abstract}

\section{Introduction}
A family  $\F$ of compact convex sets in  $\R^2$
is called a ${\mathcal P}(k)$-family if
$|\F|=2k-1$ and   no point of the plane is
contained in more than $k$ members of $\F$. 
A ${\mathcal P}(1)$-family consists of one set, 
thus  a ${\mathcal P}(1)$-family is trivially contained in any ${\mathcal P}(2)$-family.
Among the five convex sets of a ${\mathcal P}(3)$-family
there are three sets with no common point; that is, a ${\mathcal P}(3)$-family
always contains a ${\mathcal P}(2)$-family. 
To see this, observe that if any three sets have a non-empty intersection then by Helly's
theorem \cite{Eck}, there would be a point common to all five sets, which is not possible in a ${\mathcal P}(3)$-family. 
Petruska asked\footnote{personal communication in the 1970s \\
-----------------------------------------\\
 * since September 1, 2019 the Alfr\'ed R\'enyi Mathematical Institute  does not belong to the Hungarian Academy of Sciences} whether 
a ${\mathcal P}(k)$-family always contains a ${\mathcal P}(k-1)$-family,  for every $k\geq 2$. Theorem \ref{k=4} answers Petruska's question affirmatively when $k=4$, but
  Theorem \ref{k=5}  provides a counterexample for $k=5$. 

Petruska's question is equivalent to the following question:
is it true that in every family $\F$ of $2k-1$ compact convex sets in the plane
with no $(k+1)$-covered point, there are two members of $\F$
whose union contains all $k$-covered points?
It is worth noting that the question with $2k$ sets has negative answer, as it is shown by the  counterexample in Proposition \ref{2kbound}. (A point is {\it $q$-covered} by $\F$ if it is contained
in at least $q$ members of $\F$.) 

\subsection{}

A family of $2k-1$ intervals with no point of the real line covered more than $k$ times has  an interval including all $k$-covered points. For a proof it is enough to consider the far left and the far right $k$-covered points. By the pigeonhole principle there is an interval containing both, thus all $k$-covered points between them. 
 Notice that the claim is false for a family of $2k$ intervals - just take $k$ copies of each of two disjoint intervals.
 
 The observation above translates into a statement about graphs: if an interval graph of order $2\omega-1$ has maximum cliques of order $\omega$, then there is a vertex common to all maximum cliques. Actually the content of 
   the Hajnal-Folkman Lemma \cite{Hajnal} is that this last statement holds for general graphs, not just interval graphs. 
 Starting from this lemma Erd\H{o}s \cite{E} proposed a plethora of extremal problems by introducing a peculiar arrow notation widening  the scope 
 to hypergraphs, and asking for a $t$-vertex  transversal 
rather than just a common vertex of the maximum cliques. 
 
 A special instance of those  `arrow-problems' concerning $3$-uniform hypergraphs is the appropriate framework for investigating Petruska's question and its possible extensions.
  For integers $t\geq 1$ and $ \omega\geq 3$, let $H$ be a $3$-uniform hypergraph with maximum clique size $\omega(H)=\omega$ such that no $t$-element set can meet all
 maximum cliques; determine the minimum order, $n(\omega,t)$, of $H$. 
 The special value $t=1$ gained  the most attention in the literature,
and estimations on $n(\omega,1)$ were obtained by several authors, see 
\cite{GYLT,SzP,Tu}. 

 Suppose now that the $3$-uniform hypergraph $H$ with the properties above is required to be $2$-representable; 
 that is the vertices of $H$ are (compact) convex sets in $\R^2$ and its edges describe the $3$-wise intersections of the members of the family. Let the minimum order 
 of such hypergraphs be denoted by $n^*(\omega,t)$. 
The results in this paper show
 that $n^*(4,2)= 8$ and  $n^*(5,2)= 9$
 (Theorems \ref{k=4sharp} and \ref{k=5sharp}). 
 Estimating $n^*(\omega,t)$, especially the gap between 
 $n^*(\omega,t)$ and $n(\omega,t)$
presents further challenges which are mentioned in Section 
 \ref {arrow}. \\

\noindent{\bf 1.2}\\

The main tool in the study of combinatorial properties of convex sets is the nerve of a family of sets describing the intersection pattern of its members. 
Here we code intersection patterns as hypergraphs. All hypergraphs are finite and the convex sets are compact in  $\R^2$.

Let $K_n^{(3)}$ be the $3$-uniform clique
on $n$ vertices (edges are all the $3$-element subsets
of an underlying vertex set of cardinality $n$).
We represent  a family of $n$ convex sets  of $\R^2$
with the vertices of  $K_n^{(3)}$;
 the $3$-wise intersections  in the family are represented
 by a two--coloring  of the edges of  $K_n^{(3)}$:
 an edge is {\it red} if the convex sets 
the three vertices represent have a common point, it is {\it blue} otherwise. 
Colorings of a $ K_n^{(3)}$ obtained in this way
will be referred as to {\it $2$-representable} or {\it convex red/blue cliques}, 
and the family is said to
represent the red/blue clique.  

In standard terms, all red edges in a convex red/blue clique together with the vertices  and all $2$-element subsets contained by the red edges
correspond to the $2$-dimensional skeleton of the nerve complex  of a family of convex sets in  $\R^2$. On the other hand, by Helly's theorem, the $3$-wise intersections in a family of convex sets (a convex red/blue clique)  fully determines the nerve  of a representing family (if one assumes that every pair appears in some triple).

A hypergraph is called red (blue) provided all of its edges
are red (blue).  
In our hypergraph model Petruska's question becomes: if a 
convex red/blue ($3$-uniform) clique of order $2k-1$ contains no red 
clique with more than $k$ vertices, then there exist two vertices such that each red $k$-clique contains at least one of them.
This claim is verified for $k=4$ in Theorem \ref{k=4}, and in Theorem \ref{k=5} a construction is presented to refute the claim for $k=5$. The proof and the construction of the counterexample use remarkable combinatorial properties of  convex sets in  $\R^2$ extending the realm of Helly's classical theorem; among others, Lov\'asz's colorful Helly theorem, the $2$-collapsibility of an abstract simplicial complex, 
or Kalai's $f$-vector characterization of convex abstract simplicial complexes. 

Most of the combinatorial properties used here will be expressed
in terms of forbidden configurations in convex red/blue cliques, thus emphasizing the obvious fact that the investigation of the intersection patterns cannot be separated from the non-intersection patterns.
 No characterization is known for the $2$-representability 
 of an abstract simplicial complex, or equivalently,  for the convexity of a red/blue clique.  In Section  \ref{forbidden} a few non-convex red/blue subconfigurations are derived from 
 a basic property 
 (Lemma \ref{hole}) of three pairwise intersecting convex sets with no common point in the plane. Related classical intersection theorems of families of convex sets in $\R^d$ are due to Klee \cite{Klee}, Berge \cite{berge}, and Ghouila-Houri \cite{ghouila}.
 
The proof of  Petruska's question for $k=4$ and the counterexample for $k=5$ are given in Sections \ref{7convex} and \ref{9convex}, respectively. The relationship of Petruska's question with hypergraph extremal problems, in particular, with special instances of Erd\H{o}s' arrow problems on $3$-uniform hypergraphs is discussed in Section \ref{arrow}.

\section{Forbidden red/blue configurations}
\label{forbidden}

In this section we collect a few red/blue subconfigurations that are forbidden from convex red/blue cliques. Some of these can be derived from the next lemma characterizing the `hole' surrounded by three pairwise intersecting convex sets in the plane with no common point.\footnote{A $d$-dimensional extension of the Lid lemma  was obtained by J. Lehel and G. T\'oth, On the hollow enclosed by $d+1$ convex sets in $\R^d$}

Let $A,B,C\subset \R^2$ be compact convex sets, and assume that any two have a non-empty intersection, but $A\cap B\cap C=\varnothing$.  We say that $A,B,C$ with this property form a {\it hole} in the plane. A compact convex set $M$ satisfying
$M\cap(X\cap Y)\neq \varnothing$, for every $X,Y\in \{A,B,C\}$, is called a {\it convex lid} on $A\cup B\cup C$. 

\begin{lemma} 
\label{hole} 
 If $A,B,C$ form a hole, then there exist unique points $p^*\in A\cap B$, $q^*\in B\cap C$, and $r^*\in C\cap A$ such that $p^*,q^*,r^*\in M$, for every  lid $M$.
\end{lemma}

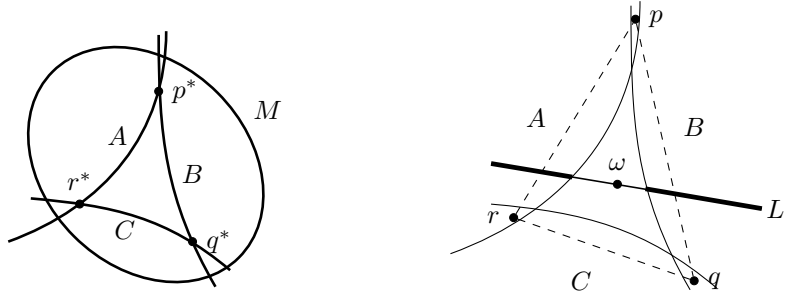
 \begin{center} 
\begin{tikzpicture}[scale=.5]
 \tikzstyle{P} = [circle, draw=black!, minimum width=3pt, inner sep=1pt, fill=black]
\tikzstyle{txt}  = [circle, minimum width=1pt, draw=white, inner sep=0pt]

\draw[line width=1pt,rotate=-45] (.5,3.4) ellipse (3.5cm and 2.7cm);
\node[P,label=right:$p^*$] (p) at (3.1,4){};
\node[P,label=above:$r^*$](r) at (1,1){};
\node[P,label=right:$q^*$](q) at (4,0){};

\draw[line width=1pt] (-.9,0) to[out=20,in=-90] (3.3,5.6); 
\draw[line width=1pt] (-.3,1.15) to[out=-5,in=140] (5,-.76); 
\draw[line width=1pt] (4.61,-1.2) to[out=120,in=-90] (3.1,5.5); 

\node ()at(6,3.5){$M$};
\node ()at(4,1.8){$B$};
\node ()at(2,2.8){$A$};
\node ()at(2.2,0.3){$C$};
\end{tikzpicture}
\hskip2cm
\begin{tikzpicture}[scale=.6]
 \tikzstyle{P} = [circle, draw=black!, minimum width=3pt, inner sep=1pt, fill=black]
\tikzstyle{txt}  = [circle, minimum width=1pt, draw=white, inner sep=0pt]

\node[P,label=right:$p$] (p) at (3.2,5.2){};
\node[P,label=left:$r$](r) at (.5,.8){};
\node[P,,label=right:$q$](q) at (4.5,-0.6){};
\draw[dashed] (p)--(q)--(r)--(p);

\draw[line width=.6] (0,2)--(6,1);
\draw[line width=1.8] (0,2)--(1.8,1.7)  (3.42,1.44)--(6,1);
\node ()at(6.3,1){$L$}; 
\draw (-.9,0) to[out=20,in=-90] (3.3,5.6); 
\draw (0,1.1) to[out=-5,in=140] (5,-.76); 
\draw (4.41,-.8) to[out=120,in=-90] (3.1,5.5); 

\node[P,label=above:$\omega$] ()at(2.8,1.54){};
\node ()at(4.5,2.8){$B$};
\node ()at(1,3){$A$};
\node ()at(2,-.6){$C$};
\end{tikzpicture}
\end{center}

\begin{proof}
(a) First we show that there is a point  $\omega\notin A\cup B\cup C$ such that  $\omega\in M$, for each convex lid $M$. 
 There is a strict separating line $L$ 
 between the two disjoint compact convex sets $A\cap B$ and $C$. 
 The closed intervals $A\cap L$ and  $B\cap L$ are non-empty and disjoint;
  therefore, they can be strictly separated by some point $\omega\in L$. 
Let  $p\in M\cap(A\cap B)$, $q\in M\cap(B\cap C)$, $r\in M\cap(C\cap A)$. The point
$\overline{pr}\cap L$ is in $(A\cap L)\cap M$ and the point
$\overline{pq}\cap L$ is in $(B\cap L)\cap M$, thus  
$\omega\in M$ follows, since $M$ is convex.\footnote{$\;\overline{xy}$ is the line segment between the points $x$ and $y$, and $\overleftrightarrow{xy}$ denotes the line through them}

Notice that $L$ intersects the boundary of both $A$ and  $B$. Furthermore, there are several separating lines through $\omega$ in the role of $L$.

(b)  Let $H\subset\R^2$ be the connected region of $\R^2\setminus(A\cup B\cup C)$ containing $\omega$. 
 Let  $K={conv}(H)$ be the convex hull of $H$, denote $\partial K$ the boundary of $K$, and set
  $cl(K)=K\cup \partial K$ for the closure of $K$.  
  Notice that $cl(K)$ is a convex set.
  
For some points
$p\in A\cap B$, $q\in B\cap C$ and $r\in C\cap A$, let $T$ be the closed triangle with these vertices. 
By part (a), $\omega \in T$, and because the sides of $T$ belong to $ A\cup B\cup C$,  $cl(K)\subseteq T$ follows. Thus we obtain that every convex lid $M$ contains $cl(K)$. 

(c)  The arguments in parts (a) and (b) show that 
 $H$ (the hole) is a bounded open region and
 $\partial H\subseteq \partial( A\cup B\cup C)$.
 Moreover, each
of the three sets $\partial A,\partial B,$ and $\partial C$ contains several 
points of $\partial H$ (cut by halflines emanating from $\omega$), and therefore,
$\partial K$ has several points in each of the sets $A, B,$ and $C$.
  
  Let $x,y\in A\cap \partial K$, and 
let  $L_A=\overleftrightarrow{xy}$. We claim that $L_A$ is a supporting line to $cl(K)$; furthermore, $A\cap \partial K$ is a line segment.

 
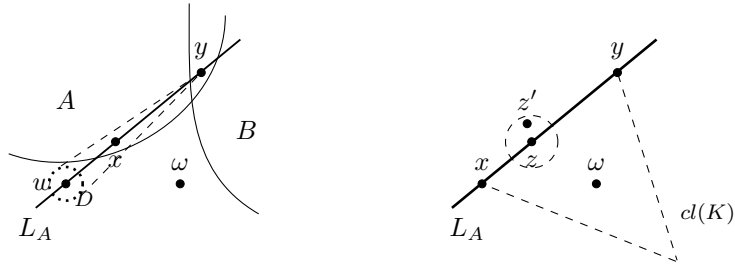
\begin{center} 
\begin{tikzpicture}[scale=.4]
 \tikzstyle{P} = [circle, draw=black!, minimum width=3pt, inner sep=1pt, fill=black]
\tikzstyle{txt}  = [circle, minimum width=1pt, draw=white, inner sep=0pt]
 \tikzstyle{O} = [circle, draw=black!, minimum width=3pt, inner sep=4.5pt]
 
 \draw (-3.9,3) to[out=-20,in=-90] (3.3,7.6); 
\node[txt,label=right:$$] (p) at (2.5,5.7){};
\node[txt,label=left:$$](r) at (-2,2){};
\node[txt,label=right:$$](q) at (4.5,-0.6){};
\draw[dashed] (p)--(-2.3,2.5) (-1.5,1.6)--(p);
\node ()at(-2,4.8){$A$};
\node ()at(4,3.8){$B$};
\node ()at(-1.4,1.5){\small $D$};
\draw[line width=.7] (-3,1.2)--(3.8,6.8) ;
\node ()at(-3,.5){$L_A$}; 
\node[P,label=above:$\omega$] ()at(1.8,2){};
\node[P,label=below:$x$] ()at(-.35,3.4){};
\node[O,dotted,line width=1] ()at (-2,2){};
\node[P,label=left:$w$] ()at(-2,2){};
\node[P,label=above:$y$] ()at(2.5,5.7){};
\draw (4.41,1) to[out=150,in=-90] (2.1,8); 
\end{tikzpicture}
\hskip2cm
\begin{tikzpicture}[scale=.4]
 \tikzstyle{P} = [circle, draw=black!, minimum width=3pt, inner sep=1pt, fill=black]
\tikzstyle{txt}  = [circle, minimum width=1pt, draw=white, inner sep=0pt]
 \tikzstyle{O} = [circle, draw=black!, minimum width=3pt, inner sep=7pt]
 
\node[txt,label=right:$$] (p) at (2.5,5.7){};
\node[txt,label=left:$$](r) at (-2,2){};
\node[txt,label=right:$$](q) at (4.5,-0.6){};
\draw[dashed] (p)--(q)--(r);

\draw[line width=1] (-3,1.2)--(3.8,6.8) ;
\node ()at(-2.5,.5){$L_A$}; 
\node[P,label=above:$\omega$] ()at(1.8,2){};
\node[P,label=above:$z^\prime$] ()at(-.5,4){};
\node[P,label=below:$z$] ()at(-.35,3.4){};
\node[O,dashed] ()at (-.35,3.4){};
\node[P,label=above:$x$] ()at(-2,2){};
\node[P,label=above:$y$] ()at(2.5,5.7){};
\node ()at(5.5,1){\small$cl(K)$};
\end{tikzpicture}
\end{center}

 Since both sets, $cl(K)$ and $A$ are convex,  $\overline{xy}\subset A\cap cl(K)$, in particular, $\overline{xy}\cap H=\varnothing$. Now assume that $L_A$ contains a point $w\in H$, and let $x\in \overline{wy}$. Because $H$ is open, 
 there is a small circular disk $D\subset H$ centered at $w$. Then $x$ is an interior point of $conv(D\cup\{y\})\subset H\subset K$, contradicting to $x\in\partial K$. Because 
 $L_A\cap H=\varnothing$ and $H$ is connected,  $H$ is on one side of $L_A$; hence $L_A$ is a supporting line to $cl(K)$.
 
 Let $z\in \overline{xy}$.  If $z\notin \partial K$, then $z\in cl(K)$ implies that $z$ is an interior point of $K$.   Then there is a point $z^\prime\in K$ sufficiently close to $z$ and on the side of  $L_A$  opposite to the one containing $H$, a contradiction. We obtain that $z\in \partial K$, for every $z\in\overline{xy}$. Because $L_A$ is a supporting line to $cl(K)$, for any choice of $x,y\in A\cap\delta K$, it follows that $A\cap \partial K$ is a line segment.\\
  
 Repeating the argument in part (c), we obtain that there are two more lines, $L_B$ and $L_C$,
 containing the segments $B\cap \partial K$  and
 $C\cap \partial K$, respectively, and supporting   $cl(K)$. 
  Therefore, 
  $cl(K)$ is contained in the intersection of the three halfplanes containing $H$ and bounded by $L_A$, $L_B$, and $L_C$.  Because each of $L_A\cap \partial K$, 
  $L_B\cap \partial K$, and $L_C\cap \partial K$
  is a line segment
  we obtain that $cl(K)$ is a triangle with vertices $p^*\in A\cap B$, $q^*\in B\cap C$, and $r^*\in C\cap A$.  By part (b),
every convex lid $M$ 
satisfies $cl(K)\subseteq M$, and hence  $p^*,q^*,r^*\in M$ follows.
\end{proof}

\begin{proposition} 
\label{H2}
If   $e$ and $e^\prime$ are blue edges in a convex red/blue clique with
$e\cap e^\prime=\{c\}$, then at least one edge in  
 the set
$\{f\;:\;|f\cap e|=2,|f\cap (e^\prime\setminus\{c\})|=1\}$ is blue.
\end{proposition}
\begin{proof}
Let the vertices in $e$ and in $e^\prime$ be represented by
the convex sets  $A,B,C$ and  $A^\prime, B^\prime, C$, respectively. Suppose 
to the contrary that every edge
$f$ with  $|f\cap e|=2$ and $|f\cap (e^\prime\setminus\{c\})|=1$ is red. Then the sets
$A,B,C$ are pairwise intersecting. Furthermore, since $e$ is blue, they form a hole. 
Now apply Lemma \ref{hole} twice with $M=A^\prime$ and $B^\prime$. It follows that
$r^*\in A^\prime\cap B^\prime\cap C$, contradicting the assumption that $e^\prime$ is blue. 
\end{proof}

 \begin{figure}[http]
  \centering
 \tikzstyle{A} = [circle, draw=black!, minimum width=1pt, inner sep=10pt]
\tikzstyle{P} = [circle, draw=black!, minimum width=3pt, inner sep=1pt, fill=black]
\tikzstyle{txt}  = [circle, minimum width=1pt, draw=white, inner sep=0pt]
\begin{tikzpicture}[scale=.8,every node/.style={draw}] 
\tikzstyle{ann} = [draw=none,fill=none,right]
\tikzstyle{txt}  = [circle, minimum width=1pt, draw=white, inner sep=0pt]

\draw[blue, line width=.5pt,rotate=15] (2.5,.8) ellipse (2cm and .4cm);

\draw[blue, line width=.5pt,rotate=-22] (1.3,3.1) ellipse (2cm and .4cm);

 \node[P](1)at (1,3){};
 \node[P](2)at (2,2.5){};
  \node[P](a)at (.7,1){};
  \node[P](b)at (2,1.5){};
\node[P,label=right:$c$](c)at (3.35,1.85){};

  \node[txt]()at (3.2,1) {$e$}; 
  \node[txt]()at (3.2,2.9) {$e^\prime$};
  
\draw[red,line width=1pt] (c)--(b)--(a) ;
\draw[red] (2)--(a)--(1)--(b)--(2)--(c) ;
\draw[red,line width=1pt] (a) to[out=5,in=-135] (c);
\draw[red] (1) to[out=0,in=130] (c);
\end{tikzpicture}
\hskip2cm
\begin{tikzpicture}[scale=.8,every node/.style={draw}] 
\tikzstyle{ann} = [draw=none,fill=none,right]
\tikzstyle{txt}  = [circle, minimum width=1pt, draw=white, inner sep=0pt]

\draw[blue, line width=.5pt] (2,1) ellipse (2cm and .4cm);
 \node[P](a)at (.7,1){};\node[P](b)at (3.3,1){};\node[P](c)at (2,1){};

\draw[blue, line width=.5pt] (2,2.5) ellipse (2cm and .4cm);
 \node[P](1)at (.7,2.5){};
 \node[P](2)at (3.3,2.5){};
 \node[P](3)at (2,2.5){};

  \node[txt]()at (4.3,1) {$e$}; 
  \node[txt]()at (4.3,2.5) {$e^\prime$};
  
\draw[red] (a)--(1)--(b) (a)--(2)--(b) (a)--(3)--(b) ;
\draw[red,line width=1pt] (a)--(b)--(c);
\draw[red] (1)--(c) (2)--(c) (3)--(c) ;
\draw[red,line width=1pt] (a) to[out=-20,in=200] (b);
\end{tikzpicture}
\caption{}
\label{lemmas12}
\end{figure}

\begin{proposition} 
\label{H1}
If $e$ and $e^\prime$ are disjoint blue edges
in a convex red/blue clique then at least
one edge in the set $\{f\;:\;|f\cap e|=2,|f\cap e^\prime|=1\}$ is blue.
\end{proposition}

Next we include two proofs of Proposition \ref{H1}. The first proof  uses the Lid lemma (Lemma \ref{hole}); the second one, showing a remarkable connection to Helly's theorem, applies the colorful Helly theorem due to Lov\'asz \cite{L}. It is stated as follows (see \cite{B} for a proof): Let $\mathcal{C}_i$, $i=1,\dots,d+1$, be  families of not necessarily distinct convex sets in $\R^d$; if for any choice of $d+1$ sets 
$K_i\in\mathcal{C}_i$, $i=1,\ldots, d+1$, we have $\bigcap\limits_{i=1}^{d+1} K_i\neq\varnothing$, then there is an index $j$, $1\leq j\leq d+1$,
such that $\bigcap\limits_{K\in \mathcal{C}_j} K\neq\varnothing$.

\begin{proof}[First proof of Proposition \ref{H1}] 
Let the vertices in $e$ and in $e^\prime$ be represented by
the convex sets  $A,B,C$ and  $A^\prime, B^\prime, C^\prime$, respectively. Suppose to the contrary that every edge
$f$ with  $|f\cap e|=2,|f\cap e^\prime|=1$ is red. Then the sets
$A,B,C$ are pairwise intersecting. Furthermore, since $e$ is blue, they form a hole. 
Now apply Lemma \ref{hole} three times with $M=A^\prime$, $B^\prime$, and $C^\prime$. It follows that
$r^*\in A^\prime\cap B^\prime\cap C^\prime$, contradicting the assumption that $e^\prime$ is blue. 
\end{proof}

\begin{proof}[Second proof of Proposition \ref{H1}] 
 The proposition follows by applying the colorful Helly theorem for $d=2$, $\mathcal{C}_1=\mathcal{C}_2=\{A,B,C\}$ and $\mathcal{C}_3=\{A^\prime,B^\prime,C^\prime\}$, where 
$A,B,C$ and 
$A^\prime, B^\prime, C^\prime$ represent  the vertices in $e$ and in $e^\prime$, respectively. Observe that any three vertices  corresponding to
a choice of $K_i\in\mathcal{C}_i$, $i=1,2,3$, define an edge $f$ of the red/blue clique such that
$|f\cap e|=2$ and $|f\cap e^\prime|=1$.  
If every such $f$ is red, the hypotheses of the colorful Helly theorem are satisfied.
Since $e$, $e^\prime$ are blue edges, $\bigcap\limits_{K\in \mathcal{C}_j} K=\varnothing$ follows for every $j=1,2,3$, contradicting the colorful Helly theorem.
\end{proof}

 The $3$--cycle, $C_3^{(3)}$,  and the chordless circular $k$-cycle $C_k^{(3)}$, for $k\geq 5$, are defined here as follows.
The {\it $3$-cycle}  $C_3^{(3)}$ has  five vertices and three pairwise intersecting  edges with no common vertex;
the {\it circular $k$-cycle}  $C_k^{(3)}$ has $k$ vertices labeled in a circular order and any three consecutive triples define an edge.

\begin{proposition}
\label{7cycle}
  A convex red/blue clique contains neither a blue $C_3^{(3)}$ nor a blue $C_k^{(3)}$, for $k\geq 6$.
\end{proposition}
\begin{figure}[htp]
\centering
 \tikzstyle{A} = [circle, draw=black!, minimum width=1pt, inner sep=10pt]
\tikzstyle{P} = [circle, draw=black!, minimum width=3pt, inner sep=1pt, fill=black]
\tikzstyle{txt}  = [circle, minimum width=1pt, draw=white, inner sep=0pt]

\begin{tikzpicture}[scale=.8]
\draw[blue, line width=1pt,rotate=15] 
(3.8,2) ellipse (1.8cm and .45cm);
\draw[blue, line width=.5pt,rotate=-12] 
(2.5,2.8) ellipse (1.6cm and .45cm);

\draw[blue, line width=1pt,rotate=-15] 
(3.8,3.5) ellipse (.45cm and 1.3cm);
       
\node[P](1)at(4.6,3.2){};         
         \node[P](3)at (4.25,2){};      
     \node[P](5)at (2.1,2.5) {};       
      \node[P](4)at (2.6,2.5){};  
      \node[P](2)at (4.8,2.5){}; 
 \node[txt]()at (3,.9) {$C_3^{(3)}$};
  \node[txt]()at (3.5,3.8) {$e^\prime$};
  \node[txt]()at (5.4,3) {$e$};
\end{tikzpicture}
\hskip2cm
\begin{tikzpicture}[scale=.8]
\draw[blue, line width=.5pt,rotate=45] 
(4.2,1.2) ellipse (1.4cm and .45cm);
\draw[blue, line width=.5pt,rotate=-45] 
(0,5.4) ellipse (1.5cm and .45cm);
\draw[blue, line width=.5pt,rotate=5] 
(2.2,2.8) ellipse (.5cm and 1.4cm);
\draw[blue, line width=.5pt,rotate=-5] 
(3.75,3.3) ellipse (.5cm and 1.4cm);

 \draw [blue,line width=.5pt,rounded corners=5mm] 
 (1.5,3.8)--(3,4.9)--(4.5,3.8)--cycle
 (1.7,1.8)--(1.2,3.5)--(4.5,1.85)--cycle 
 (1.3,1.75)--(4.8,3.5)--(4.3,1.6)--cycle;
 
\node[P](0)at(3,4.5){}; 
\node[P](6)at(2,4){};
 \node[P](1)at(4,4){};     
 \node[P](5)at (1.6,3) {};       
  \node[P](2)at (4.4,3){};  
  \node[P](4)at (2.25,2){};  
\node[P](3)at (3.75,2){}; 

 \node[txt]()at (3,.9) {$C_7^{(3)}$};
 
\end{tikzpicture}

\caption{}
\label{obst}
\end{figure}

\begin{proof}  
Apply Proposition \ref{H2} to the edges $e, e^\prime$ of $C_3^{(3)}$ in Fig.\ref{obst}; besides the edges of $C_3^{(3)}$ there must exist another blue edge. 
 For $k\geq 8$, let 
$e, e^\prime$ be disjoint edges of $C_k^{(3)}$ separated by at least one vertex on each side, 
 then apply Proposition \ref{H1}.  Besides the edges of $C_k^{(3)}$ there must exist another blue edge. 
The cases $k=6$ and $7$ follow from the more general statement 
that $C_k^{(3)}$ is not $2$-collapsible for $k\geq 6$.

A face $F$ of a simplicial complex is $d$-collapsible if it is contained by a unique maximal face
of dimension less than $d$. The removal of  $F$ and all faces containing $F$ is called an elementary collapse. A simplicial complex is $d$-collapsible if there is a sequence of elementary collapses ending with an empty complex (for more details and examples see \cite{T}).  Wegner \cite{W} proved  that $d$-representable simplicial complexes are
$d$-collapsible. 
The red subhypergraph in the complement of a chordless circular blue $C_k^{(3)}$
has no elementary $2$-collapse, since every vertex and every edge belongs to more than one maximal red clique, for $k\geq 6$.
\end{proof}

\section{The seven convex sets theorem}
\label{7convex}
In this section we prove the following theorem.
\begin{theorem} 
\label{k=4} 
If a convex red/blue $K_7^{(3)}$ 
contains no red   $K_5^{(3)}$,  then there exist two vertices
such that every red $K_4^{(3)}$ contains at least one of them.
\end{theorem}
The proof of Theorem \ref{k=4} appears at the end of the section. First we develop some tools.
We say that a red/blue clique contains a blue partial hypergraph $H$, provided that a subset of blue edges is isomorphic to $H$.
Let $V$ be the vertex set of a red/blue $K_7^{(3)}$.
A $2$-vertex set $\{x,y\}\subset V$   will be called a {\it $K_4^{(3)}$-transversal}
if all  copies of a red $K_4^{(3)}$ in 
$K_7^{(3)}$ contains at least one of $x$ or $y$.
\begin{proposition}\label{C3}
If an arbitrary (not necessarily convex)  red/blue $K_7^{(3)}$ contains a blue  $C_3^{(3)}$, then
there exists a $K_4^{(3)}$-transversal.
\end{proposition}
 \begin{proof} Let $X=e_1\cup e_2\cup e_3$, where 
 $e_1,e_2, e_3$ are the edges of the blue $C_3^{(3)}$.
 The set $V\setminus X$
is a $K_4^{(3)}$-transversal, since every red 
$K_4^{(3)}$ contains at most two vertices from $X$.
\end{proof}

\begin{proposition}\label{V}
Let $e_1, e_2$ be blue edges of a convex red/blue $K_7^{(3)}$ and suppose that $e_1\cap e_2=\{c\}$. If  $K_7^{(3)}$ does not have 
a $K_4^{(3)}$-transversal, then every blue edge $e_3\subset e_1\cup e_2$ contains $c$.
\end{proposition} 
\begin{proof} By Proposition \ref{H2}, there is a blue edge
$e_3\subset e_1\cup e_2$ satisfying $|e_3\cap e_1|=2$ and
$|e_3\cap (e_2\setminus\{c\})|=1$. If $c\notin e_3$, then 
$e_1, e_2, e_3$ form a blue $3$-cycle. Then, by Proposition \ref{C3}, there is a $K_4^{(3)}$-transversal; thus the contradiction implies $c\in e_3$.
\end{proof}

The {\it transversal number} $\tau=\tau(H)$ of a hypergraph $H=(V,E)$ is defined as
the minimum cardinality of a set $T\subset V$ such that $e\cap T\neq\varnothing$, for every edge $e\in E$. 

\begin{figure}[htp]
 \centering
  \tikzstyle{A} = [circle, draw=black!,minimum width=2pt, inner sep=2.5pt]
\tikzstyle{P} = [circle,draw=black, minimum width=3pt, inner sep=1pt, fill=black]
\tikzstyle{txt}  = [circle, minimum width=1pt, draw=white, inner sep=0pt]

\begin{tikzpicture}[scale=.1]

  \node[P,label=below:{\small$3$}](3)at (0,0){}; 
   \node[P,label=right:{\small$2$}](2)at (3,5){};
      \node[P,label=above:{\small$1$}](1)at (0,10){};      
    \node[P,label=left:{\small$0$}](0)at (-3,5){};   
          
     \node[P,label=right:$4$](a)at (12,12){};          
     \node[P,label=right:$5$](b)at (12,5){}; 
      \node[P,label=right:$6$](c)at (12,-2){}; 
        
  \draw[]  (0,5) ellipse (5 and 6.7);
 \draw[]  (12,5) ellipse (1.5 and 10);
 
  \node[txt](A)at (7,-10){A}; 
  \node[txt](K4)at (0,19){K$_0$};
  \node[txt](K3)at  (12,19){e$_0$};
\end{tikzpicture}
\hskip1cm
\begin{tikzpicture}[scale=.1]
             \node[P,label=above:$$](1)at (0,22){};  
   \node[P,label=above:$$](2)at (3,27){};
      \node[P,label=above:$$](3)at (0,32){};      
    \node[P,label=above:$$](4)at (-3,27){};   
          
     \node[P,label=above:$$](a)at (12,34){};          
     \node[P,label=above:$$](b)at (12,27){}; 
      \node[P,label=above:$$](c)at (12,20){}; 
        
        \draw[]      (2)--(a)--(b)--(2);
  \draw[]  (0,27) ellipse (6 and 7);
 \draw[]  (12,27) ellipse (2 and 10);
 
  \node[P,label=above:$$](1)at (0,0){}; 
   \node[P,label=above:$$](2)at (3,5){};
      \node[P,label=above:$$](3)at (0,10){};      
    \node[P,label=above:$$](4)at (-3,5){};   
          
     \node[P,label=above:$$](a)at (12,12){};          
     \node[P,label=above:$$](b)at (12,5){}; 
      \node[P,label=above:$$](c)at (12,-2){}; 
        
               \draw[]      (2)--(a)--(b)--(2);
        \draw[]      (2)--(c)--(b)--(2); 
  \draw[]  (0,5) ellipse (6 and 7);
 \draw[]  (12,5) ellipse (2 and 10);
 
   \node[txt](A4)at (20,22){A$_1$}; 
 \node[txt](A4)at (20,0){A$_3$}; 
  \end{tikzpicture}
\hskip.5cm
\begin{tikzpicture}[scale=.1]
   \node[P,label=above:$$](1)at (0,22){};  
   \node[P,label=above:$$](2)at (3,27){};
      \node[P,label=above:$$](3)at (0,32){};      
    \node[P,label=above:$$](4)at (-3,27){};   
          
     \node[P,label=above:$$](a)at (12,34){};          
     \node[P,label=above:$$](b)at (12,27){}; 
      \node[P,label=above:$$](c)at (12,20){}; 
      
         \draw[]      (2)--(a)--(b)--(2);
        \draw[]      (1)--(c)--(b)--(1);  
  \draw[]  (0,27) ellipse (6 and 7);
 \draw[]  (12,27) ellipse (2 and 10);

  \node[P,label=above:$$](1)at (0,0){}; 
   \node[P,label=above:$$](2)at (4,5){};
      \node[P,label=above:$$](3)at (0,10){};      
    \node[P,label=above:$$](4)at (-3,5){};   
          
     \node[P,label=above:$$](a)at (8.6,10){};          
     \node[P,label=above:$$](b)at (10,5){}; 
      \node[P,label=above:$$](c)at (8.6,0){}; 
        
  \draw[]  (0,5) ellipse (6 and 7);
 \draw[]  (7,5) ellipse (5 and 9);
 
  \node[txt](A4)at (20,22){A$_2$}; 
 \node[txt](A4)at (20,0){A$_4$}; 
  \node[txt](K4)at (16,11){$K_4^{(3)}$};
   \end{tikzpicture}
\hskip1cm
\begin{tikzpicture}[scale=.1]

  \node[P,label=above:$$](1)at (0,-1){}; 
   \node[P,label=above:$$](2)at (3,5){};
      \node[P,label=above:$$](3)at (0,10){};      
    \node[P,label=above:$$](4)at (-3,5){};   
          
     \node[P,label=above:$$](a)at (12,12){};          
     \node[P,label=above:$$](b)at (15,5){}; 
      \node[P,label=above:$$](c)at (12,-2){}; 
        
    \draw[]      (3)--(a)--(b)--(3);
        \draw[]      (1)--(c)--(b)--(1);  
          \draw[]      (1,5)--(13,-4)--(13,14)--(1,5);  
  \draw[] (0,5) ellipse (6 and 8);
 \draw[] (12.5,5) ellipse (4 and 11);
 
 \node[txt](B)at (7,-10){B}; 
   \node[txt](K4)at (0,17){K$_0$};
    \end{tikzpicture}
   \hskip.5cm
\begin{tikzpicture}[scale=.1]

  \node[P,label=above:$$](1)at (0,-1){}; 
   \node[P,label=above:$$](2)at (3,5){};
      \node[P,label=above:$$](3)at (0,10){};      
    \node[P,label=above:$$](4)at (-3,5){};   
          
     \node[P,label=above:$$](a)at (12,12){};          
     \node[P,label=above:$$](b)at (16,5){}; 
      \node[P,label=above:$$](c)at (12,-2){}; 
        
  \draw[] (0,5) ellipse (6 and 8);
       
    \draw[]      (3)--(a)--(b)--(3);
        \draw[]      (1)--(c)--(b)--(1);  
          \draw[]      (1,5)--(13,-4)--(13,14)--(1,5); 
 
 \node[txt](AB-)at (7,-10){B$^-$}; 
   \end{tikzpicture}
\vskip.5cm
\caption{}
\label{Aderiv}
\end{figure}

\begin{proposition}
\label{tauK4}
There are seven  $3$-uniform hypergraphs on $7$ vertices with $\tau\geq 3$ that  do not contain $C_3^{(3)}$, but do contain $K_4^{(3)}$:  the hypergraph $A=K_4^{(3)}+K_3^{(3)}$, four extensions of $A$, and the hypergraphs $B$, $B^-$ (see Fig.\ref{Aderiv}).
\end{proposition}
\begin{proof} Let $H=(V,E)$ be a $3$-uniform hypergraph with $|V|=7,$ $\tau\geq 3$, and containing a $4$-clique $K_0\subset H$. Label the vertices of $K_0$ with $0,1,2,3$ and let $4,5,6$ be the labels of the vertices not in $K_0$. 

Assume first that  $e_0=\{4,5,6\}$ is an edge of $H$, which implies $\tau(H)=3$. If $H$ has no more edges, then it has no $3$-cycle either, hence $H\cong A$. Next we add edges to $A$ with avoiding a $3$-cycle. Each additional edge 
has at least one vertex in $K_0$. Furthermore, because $C_3^{(3)}$ must be avoided, no two additional edges share a common pair in $e_0$. Adding one or two edges to $A$ we obtain that $H\cong A_i$ for $i=1,2$ or $3$. 

 Observe  that $\{2,4,5\}$, $\{2,5,6\},$ and $\{3,4,6\}$ form a $3$-cycle on vertex set $\{2,3,4,5,6\}$. Therefore, when 
three edges are added to $A$,  either they have a common vertex in $K_0$ producing $H\cong A_4$, or they intersect $K_0$ in three distinct vertices, which leads to $H\cong B$.

Assume now that $\{4,5,6\}$ is not an edge of $H$.  There are two vertices of $K_0$ that form a $2$-vertex transversal set of $H$, unless at least  three vertices of $K_0$  are covered by some edge not in $K_0$.  This leads to $H\cong B^-$.
\end{proof}

\begin{figure}[htp]
\centering
  \tikzstyle{A} = [circle, draw=black!,minimum width=2pt, inner sep=2.5pt]
\tikzstyle{P} = [circle, draw=black!, minimum width=3pt, inner sep=1pt, fill=black]
\tikzstyle{txt}  = [circle, minimum width=1pt, draw=white, inner sep=0pt]
\begin{tikzpicture}[scale=.5]
\draw[line width=.5pt,rotate=45] 
(4.2,1.2) ellipse (1.4cm and .45cm);
\draw[line width=.5pt,rotate=-45] 
(0,5.4) ellipse (1.5cm and .45cm);
\draw[line width=.5pt,rotate=5] 
(2.2,2.8) ellipse (.5cm and 1.4cm);
\draw[line width=.5pt,rotate=-5] 
(3.75,3.3) ellipse (.5cm and 1.4cm);

 \draw [line width=.5pt,rounded corners=5mm] 
 (1.5,3.8)--(3,4.9)--(4.5,3.8)--cycle
 (1.7,1.8)--(1.2,3.5)--(4.5,1.85)--cycle 
 (1.3,1.75)--(4.8,3.5)--(4.3,1.6)--cycle;
 
\node[P](0)at(3,4.5){}; 
\node[P](6)at(2,4){};                 
\node[P](1)at(4,4){};
     
\node[P](5)at (1.6,3) {};       
\node[P](2)at (4.4,3){};
   
\node[P](4)at (2.25,2){};  
\node[P](3)at (3.75,2){}; 

 \node[txt]()at (3,1) {$C_7^{(3)}$};
\end{tikzpicture}
\hskip.3cm
\begin{tikzpicture}[scale=.08]
 
      \node[txt](x1)at (10,-2){}; 
     \node[txt](y1)at (22.3,25){};     
     \draw  []   (x1)--(y1);
    
     \node[txt](x2)at (-2,-2){}; 
     \node[txt](y2)at (24,24){};     
     \draw  (x2)--(y2);  
     
       \node[txt](u2)at (25,1){}; 
     \node[txt](v2)at (-3,1){};     
     \draw   (u2)--(v2);  
     
       \node[txt](u1)at (22,-1){}; 
     \node[txt](v1)at (9,18){};      
    \draw   (u1)--(v1);
   
       \node[txt](z1)at (-3,-1){}; 
     \node[txt](w1)at (25.6,14){};      
     \draw []   (z1)--(w1);
      
        \node[txt](z2)at (20.6,-2){}; 
     \node[txt](w2)at (20.1,24){};     
     \draw   (z2)--(w2);  
     
  \node[P,label=left:$$](p)at (20,20){}; 
  \node[P,label=above:$$](q)at (20.5,1){}; 
  
     \node[P](v1)at (11.5,1){};  
    \node[P,label=below:$$](c)at (1,1){};
      
   \node[P,label=above:$$](C)at (12.6,12.6){}; 
       \node[P](v2)at (15.2,8.6){};    
  \node[P,](v3)at (20.3,11.1){}; 
      
  \draw[] (15,13.7) to[out=180,in=180] (v1);
\draw (19,12.5) to[out=-30,in=10] (v1);
             \node[txt]()at (15,-5){F};
\end{tikzpicture}
\hskip.3cm
\begin{tikzpicture}[scale=.08]
 
      \node[txt](x1)at (10,-2){}; 
     \node[txt](y1)at (22.3,25){};     
     \draw  []   (x1)--(y1);
    
     \node[txt](x2)at (-2,-2){}; 
     \node[txt](y2)at (24,24){};     
     \draw  (x2)--(y2);  
     
       \node[txt](u2)at (25,1){}; 
     \node[txt](v2)at (-3,1){};     
     \draw   (u2)--(v2);  
     
       \node[txt](u1)at (22,-1){}; 
     \node[txt](v1)at (9,18){};      
    \draw   (u1)--(v1);
   
       \node[txt](z1)at (-3,-1){}; 
     \node[txt](w1)at (25.6,14){};      
     \draw []   (z1)--(w1);
      
        \node[txt](z2)at (20.6,-2){}; 
     \node[txt](w2)at (20.1,24){};     
     \draw   (z2)--(w2);  
     
  \node[P,label=left:$$](p)at (20,20){}; 
  \node[P,label=above:$$](q)at (20.5,1){}; 
  
     \node[P](v1)at (11.5,1){};  
    \node[P,label=below:$$](c)at (1,1){};

   \node[P,label=above:$$](C)at (12.6,12.6){}; 
       \node[P](v2)at (15.2,8.6){};
    
  \node[P,](v2)at (20.3,11.1){}; 
      
             \node[txt]()at (15,-5){F$^-$};
\end{tikzpicture}
\hskip.3cm
\begin{tikzpicture}[scale=.08]
 
      \node[txt](x1)at (2,0){}; 
     \node[txt](y1)at (14,24){};     
     \draw[]    (x1)--(y1);
     \node[txt](x2)at (12,-1){}; 
     \node[txt](y2)at (12,26){};     
     \draw[]    (x2)--(y2);  
       \node[txt](x3)at (27,0){}; 
     \node[txt](y3)at (9,24){};     
     \draw[]    (x3)--(y3);  
     
       \node[txt](u1)at (-1,4){}; 
     \node[txt](v1)at (30,4){};     
     \draw[]    (u1)--(v1);
     \node[txt](u2)at (4,12){}; 
     \node[txt](v2)at (24,12){};     
     \draw[]    (u2)--(v2);  
     
  \node[P,label=above:$$](a1)at (4,4){}; 
   \node[P,label=above:$$](b1)at (8,12){};  
     \node[P,label=above:$$](b)at (12,20){};  
    
    \node[P,label=above:$$](a2)at (12,4){};   
     \node[P,label=above:$$](b2)at (12,12){}; 
     
    \node[P,label=above:$$](a3)at (24,4){};    
     \node[P,label=above:$$](b3)at (18,12){}; 
       \node[txt](C)at (15,-4){C};
       
\end{tikzpicture}
\hskip.3cm
\begin{tikzpicture}[scale=.08]
 
      \node[txt](x1)at (2,0){}; 
     \node[txt](y1)at (14,24){};     
     \draw[]    (x1)--(y1);
      
       \node[txt](x3)at (27,0){}; 
     \node[txt](y3)at (9,24){};     
     \draw[]    (x3)--(y3);  
     
       \node[txt](u1)at (-1,4){}; 
     \node[txt](v1)at (30,4){};     
     \draw[]    (u1)--(v1);
  
  \node[P,label=above:$$](A)at (4,4){}; 
   \node[P,label=above:$$](B)at (24,4){};
      \node[P,label=above:$$](C)at (12,20){};   
    
    \node[P,label=above:$$](C1)at (14,4){};         
     \node[P,label=above:$$](B1)at (8,12){};          
     \node[P,label=above:$$](A1)at (18,12){}; 
     
  \draw[,rotate=45] 
(15.7,-5) ellipse (8 and 4);
  \draw (13,11) ellipse (7 and 4);
    \draw[,rotate=-45] 
(3,14) ellipse (9 and 4);
  \node[P,label=above:$$](X)at (12.8,8.8){}; 
  \node[txt](D)at (15,-4){D}; 
  \end{tikzpicture}
\hskip.3cm
\begin{tikzpicture}[scale=.08]
 
      \node[txt](x1)at (2,0){}; 
     \node[txt](y1)at (14,24){};     
     \draw[]   (x1)--(y1);
      
       \node[txt](x3)at (27,0){}; 
     \node[txt](y3)at (9,24){};     
     \draw[]    (x3)--(y3);  
     
       \node[txt](u1)at (-1,4){}; 
     \node[txt](v1)at (30,4){};     
     \draw[]    (u1)--(v1);
  
  \node[P,label=above:$$](A)at (4,4){}; 
   \node[P,label=above:$$](B)at (24,4){};
      \node[P,label=above:$$](C)at (12,20){};   
    
    \node[P,label=above:$$](C1)at (14,4){};         
     \node[P,label=above:$$](B1)at (8,12){};          
     \node[P,label=above:$$](A1)at (18,12){}; 
       \node[P,label=above:$$](X)at (12.8,8.8){}; 
       
  \draw[,rotate=45] 
(15.7,-5) ellipse (8 and 4);
  \draw[]  (13,11) ellipse (7 and 4);
    \draw[,rotate=-45] 
(3,14) ellipse (9 and 4);

  \draw[] (5.8,2.5) to[out=160,in=170] (12,20);
    \draw (22.8,2.5) to[out=50,in=-10] (12,20);
  \node[txt](D)at (15,-4){D$^+$}; 
  \end{tikzpicture}
\vskip.5cm
\caption{}
\label{CDetc}
\end{figure}

\begin{proposition}
\label{taunoK4}
There are six  $3$-uniform hypergraphs on $7$ vertices with $\tau\geq 3$ that  contain neither $C_3^{(3)}$ nor  $K_4^{(3)}$; these hypergraphs are 
  $C_7^{(3)}$,  $F$ (the Fano-plane), $F^-$ (the Fano-plane minus a line), 
 $C,D,$ and  $D^+$ (see Fig.\ref{CDetc}).
\end{proposition}
\begin{proof}  Let $H=(V,E)$ be a $3$-uniform $K_4^{(3)}$- and $C_3^{(3)}$-free hypergraph with $|V|=7$ and $\tau(H)\geq 3$. Because $\tau(H)\geq 3$ and 
$H$ is $C_3^{(3)}$-free, we have the following.

\noindent {\bf Observation}
\label{H-w}
{\it For every $w\in V$ the partial hypergraph $H-w$ 
(obtained from $H$ by removing  $w$ and all incident edges) has either two independent edges or 
a  triangle shown in Fig.\ref{triangle}.
}\\

First assume that any two edges of $H$ have a common vertex.
 By Observation \ref{H-w}, $H$ has no vertex of degree $0$ or $1$. Since  the degree sum must be a multiple of $3$, there exists a vertex $w$ contained by three (or more) edges.  
Because $\tau(H)\geq 3$ and there is no $C_3^{(3)}$, we have $|e_1\cap e_2|=1$, for any $e_1\neq e_2$. Therefore $H$ has a spanning $3$-star $S_0$. Removing the center $w$ of $S_0$ we obtain a spanning triangle $T_0$ such that no edge of $T_0$ has two common vertices with any edge of $S_0$. The essentially unique placement of $T_0$ into $S_0$ yields the lines of $F^-$ (a Fano plane with one line removed). Only this line missing from the Fano-plane can be added without creating a $C_3^{(3)}$.

Suppose now that $H$ has a spanning $3$-star with center $w$ and $H-w$ contains two independent edges. Then $H\cong C$, moreover, no edge can be added to $H$ without creating $C_3^{(3)}$.

From now on we assume that $H$ has independent edges and has no spanning $3$-star. Suppose that $H$ has a {\it $3$-wheel} $W_0$ centered at $w$ (defined as a $4$-clique minus the  edge not containing $w$). By Observation \ref{H-w}, there is a pair of disjoint edges or a triangle in $H-w$. Since an edge containing two vertices  of $H-w$ forms a $C_3^{(3)}$, 
there is a triangle $T_0$ spanning $H-w$. Observe that
the `corner vertices` of $T_0$ are not in $H-w$, by the same argument. Thus either $H\cong D$ or one more edge can be included leading to $D^+$.
 \begin{figure}[http]
 \centering
  \tikzstyle{A} = [circle, draw=black!,minimum width=2pt, inner sep=2.5pt]
\tikzstyle{P} = [circle, draw=black!, minimum width=3pt, inner sep=1pt, fill=black]
\tikzstyle{txt}  = [circle, minimum width=1pt, draw=white, inner sep=0pt]
\tikzstyle{B} = [rectangle, draw=black!,minimum width=2pt, inner sep=2.5pt]

\begin{tikzpicture}[scale=.1]    
    
       \draw[] (5,12) ellipse (9 and 1.5);
        \draw[rotate=45] (11,1) ellipse (8 and 1.5);
 \draw[rotate=-45] (-4,8) ellipse (8 and 1.5);
 
   \node[P,label=above:$$](b1)at (10,12){}; 
        \node[P,label=above:$$](b2)at (5,12){};   
        \node[P,label=above:$$](b3)at (0,12){}; 
        
     \node[P,label=above:$w$](w)at (5,20){};  
    
    \node[P,label=above:$$](a2)at (5,6){};   
       \node[P,label=above:$$](a1)at (2.6,9){}; 
    \node[P,label=above:$$](a3)at (7.4,9){};

         \node[txt]()at (5,29){\small triangle};
       
\end{tikzpicture}
\hskip1.5cm
\begin{tikzpicture}[scale=.1]    
    
       \draw[] (5,13) ellipse (1.5 and 9);
        \draw[rotate=45] (10,10) ellipse (10 and 1.5);
 \draw[rotate=-45] (-5,17) ellipse (10 and 1.5);
 
         \node[P,label=above:$$](b1)at (9.6,15){}; 
        \node[P,label=above:$$](b2)at (5,14){};   
        \node[P,label=above:$$](b3)at (0.4,15){}; 
     \node[P,label=above:$w$](w)at (5,19.5){};  
        \node[P,label=above:$$](a1)at (13,11){}; 
        \node[P,label=above:$$](a2)at (5,9){};   
        \node[P,label=above:$$](a3)at (-3.5,11){}; 

        \node[txt]()at (5,31){\small spanning $3$-star};
\end{tikzpicture}
\hskip1.5cm
\begin{tikzpicture}[scale=.1]
 
      \node[txt](x1)at (10,-2){}; 
     \node[txt](y1)at (22.3,25){};     
     \draw  []   (x1)--(y1);
    
     \node[txt](x2)at (-2,-2){}; 
     \node[txt](y2)at (24,24){};     
     \draw  (x2)--(y2);  
     
       \node[txt](u2)at (25,1){}; 
     \node[txt](v2)at (-3,1){};     
     \draw   (u2)--(v2);  
     
       \node[txt](u1)at (22,-1){}; 
     \node[txt](v1)at (9,18){};      
    \draw   (u1)--(v1);
   
       \node[txt](z1)at (-3,-1){}; 
     \node[txt](w1)at (25.6,14){};      
     \draw []   (z1)--(w1);
      
        \node[txt](z2)at (20.6,-2){}; 
     \node[txt](w2)at (20.1,24){};     
     \draw   (z2)--(w2);  
     
  \node[P,label=left:$w$](p)at (20,20){}; 
  \node[P,label=above:$$](q)at (20.5,1){}; 
  
     \node[P](v1)at (11.5,1){};  
    \node[P,label=below:$$](c)at (1,1){};

   \node[P,label=above:$$](C)at (12.6,12.6){}; 
       \node[P](v2)at (15.2,8.6){};
    
  \node[P,](v2)at (20.3,11.1){}; 
      
      \node[txt]()at (3,25){\small Fano minus a line};
        
\end{tikzpicture}
\hskip1.5cm
\begin{tikzpicture}[scale=.08]    
    
       \draw[rotate=45] (10,10) ellipse (12 and 8);
        \draw[rotate=-45] (-4,15) ellipse (12 and 8);
            \draw[] (5,8) ellipse (12 and 8);
  \node[P,label=above:$$](c)at (15,8){};  
        \node[P,label=above:$$](a)at (-5,8){};   
        \node[P,label=left:$$](w)at (5,12){};     
         \node[txt]()at (2.3,12){\small $w$};    
     \node[P,label=above:$$](b)at (5,20){};         
       
         \node[txt]()at (5,33){\small $3$-wheel};
       
\end{tikzpicture}
\vskip.5cm
\caption{}
\label{triangle}
\end{figure}

Excluding the previous cases we assume that $H$ has 
neither a spanning $3$-star, nor a $3$-wheel, and nor a
$K_4^{(3)}$. 
Of course, we also have that $H$ is $C_3^{(3)}$-free and $\tau(H)\geq 3$. We claim that $H\cong C_7^{(3)}$.

Suppose that $H$ has three distinct edges $e_1,e_2,e_3$ such that $e_1\cap e_2\cap e_3=\{w,w^\prime\}$.  Observe that each edge in $H-w$ with two vertices in 
$(e_1\cup e_2\cup e_3)\setminus \{w\}$ forms a $C_3^{(3)}$ or a $3$-wheel, hence $H-w$ does not contain independent edges.  
Then, by Observation \ref{H-w}, $H-w$ has a spanning triangle $T_0$ with edges $f_1, f_2, f_3$; morever, $T_0$ has no  corner vertex in  $(e_1\cup e_2\cup e_3)\setminus \{w^\prime\}$. Without loss of generality we may assume that $f_1\cap f_2=\{w^\prime\}$  and 
$|f_1\cap e_1|=|f_2\cap e_2|=2$. This means that $f_1,f_2$ and $e_3$ is a spanning $3$-star of $H$, a contradiction.

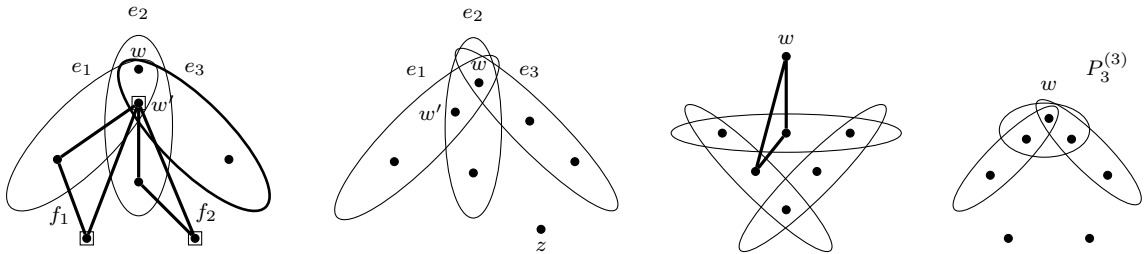
\begin{figure}[http]
\centering
  \tikzstyle{A} = [circle, draw=black!,minimum width=2pt, inner sep=2.5pt]
\tikzstyle{P} = [circle, draw=black!, minimum width=3pt, inner sep=1pt, fill=black]
\tikzstyle{txt}  = [circle, minimum width=1pt, draw=white, inner sep=0pt]
\tikzstyle{B} = [rectangle, draw=black!,minimum width=2pt, inner sep=2.5pt]
\begin{tikzpicture}[scale=.15]    
    
       \draw[] (5,15) ellipse (3 and 8);
        \draw[rotate=45] (10,10) ellipse (9 and 3);
 \draw[rotate=-45,line width=1] (-3,17) ellipse (9 and 3);

        \node[P,label=right:$w^\prime$](w')at (5,17){};     \node[B]()at (5,17){};     
     \node[P,label=above:$w$](w)at (5,20){};  
      
        \node[P,label=above:$$](a1)at (13,12){}; 
        \node[P,label=above:$$](a2)at (5,10){};   
        \node[P,label=above:$$](a3)at (-2.2,12){}; 
        
          \node[P,label=below:$$](b2)at (0.4,5){};    \node[B]()at (.4,5){}; 
          \node[P,label=below:$$](b1)at (10,5){};  \node[B]()at (10,5){}; 
           \node[txt]()at (10,20){$e_3$};
       \node[txt]()at (5,25){$e_2$};
        \node[txt]()at (0,20){$e_1$};  
        
        \draw[line width=1.2]   (w')--(a2)--(b1)--(w');
         \draw[line width=1.2]   (w')--(a3)--(b2)--(w');
            \node[txt]()at (11,7.3){$f_2$};
        \node[txt]()at (-2,7){$f_1$};  
\end{tikzpicture}
\hskip.5cm
\begin{tikzpicture}[scale=.15]    
    
       \draw[] (5,15) ellipse (2.5 and 8);
        \draw[rotate=45] (10,10) ellipse (10 and 2.5);
 \draw[rotate=-45] (-3,18) ellipse (10 and 2);
 
         \node[P,label=above:$$](b1)at (10,15.6){}; 
        \node[P,label=left:$w^\prime$](w')at (3.4,16.4){};   
    
     \node[P,label=above:$w$](w)at (5.5,19){};  
     
        \node[P,label=above:$$](a1)at (14,12){}; 
        \node[P,label=above:$$](a2)at (5,11){};   
        \node[P,label=above:$$](a3)at (-2,12){}; 
        
           \node[P,label=below:$z$](z)at (11,6){}; 
           \node[txt]()at (10,20){$e_3$};
       \node[txt]()at (5,25){$e_2$};
        \node[txt]()at (0,20){$e_1$};  
\end{tikzpicture}
\hskip.5cm
\begin{tikzpicture}[scale=.17]    
    
       \draw[] (5,12) ellipse (9 and 1.5);
        \draw[rotate=45] (11,1) ellipse (8 and 1.5);
 \draw[rotate=-45] (-4,8) ellipse (8 and 1.5);
 
   \node[P,label=above:$$](b1)at (10,12){}; 
        \node[P,label=above:$$](b2)at (5,12){};   
        \node[P,label=above:$$](b3)at (0,12){}; 
        
     \node[P,label=above:$w$](w)at (5,18){};  
    
    \node[P,label=above:$$](a2)at (5,6){};   
       \node[P,label=above:$$](a1)at (2.6,9){}; 
    \node[P,label=above:$$](a3)at (7.4,9){};   
     
           \draw[line width=1.2]   (w)--(b2)--(a1)--(w);
       
\end{tikzpicture}
\hskip.5cm
\begin{tikzpicture}[scale=.12]    
    
       \draw[] (5,18) ellipse (5 and 3);
        \draw[rotate=45] (11,10) ellipse (8 and 2);
 \draw[rotate=-45] (-4,18) ellipse (8 and 2);
 
         \node[P,label=above:$$](b1)at (8,17){}; 
        \node[P,label=left:$$]()at (3,17){};   
    
     \node[P,label=above:$$](w)at (5.5,19.3){};  
     \node[txt]()at (5.5,23){$w$};
        \node[P,label=above:$$](a1)at (12,13){}; 
        \node[P,label=above:$$](a2)at (1,6){};   
        \node[P,label=above:$$](a3)at (-1,13){}; 
        
           \node[P,label=below:$$]()at (10,6){}; 
           \node[txt]()at (12,25){\small $P_3^{(3)}$};
    
\end{tikzpicture}
\end{figure}

Suppose now that $H$ has distinct edges $e_1,e_2,e_3$ such that
 $e_1\cap e_2=\{w,w^\prime\}$ and $e_1\cap e_2\cap e_3=\{w\}$. 
 By Observation  \ref{H-w}, $H-w$ has independent edges $f_1,f_2$ or three edges of a triangle, $f_1,f_2,f_3$. Observe that 
edges induced by  $ (e_1\cup e_2\cup e_3)\setminus\{w\}$ produce
 a $C_3^{(3)}$ or a $3$-wheel with two edges from $e_1,e_2,e_3$. 
 Note that there exists a vertex $z$ that is not in the union 
  $e_1\cup e_2\cup e_3$. Now $z$  belongs to at most one among $f_1,f_2$ or at most two among $f_1,f_2,f_3$, so  $f_i\subset (e_1\cup e_2\cup e_3)\setminus\{w\}$, for some $i=1,2$ or $3$, a contradiction.

The argument above shows that if a vertex belongs to three edges of $H$, then one of them is contained in the union of the other two. 
Due to this property, if a triangle spans $H-w$ then any edge $f$ through $w$ cannot contain a corner vertex of the triangle. A second edge through $w$ would result a common vertex $w^\prime$, neither contained by the union of the other two, a contradiction.

For $k\leq 6$, let the $k$-path $P_k^{(3)}$ be defined here as the subpath containing $k$ consecutive edges of a $C_7^{(3)}$. 
Since each of the $7$ vertices of $H$ has degree at least two, and the sum of the degrees is a multiple of $3$, there is a vertex $w$ contained by three (or more) edges; these three form a $3$-path $P_3^{(3)}$. Then $H-w$ 
is spanned by two independent edges, the only possibility to avoid a $C_3^{(3)}$ and a triangle is extending the $3$-path at both ends to a $5$-path. Repeating the same argument $P_3^{(3)}$ closes to a $7$-cycle. No further triples can be added to $C_7^{(3)}$ without creating a $3$-cycle, thus  $H\cong C_7^{(3)}$ follows. 
\end{proof}

\noindent{\it Proof of Theorem \ref{k=4}.} We assume that the convex red/blue $K_7^{(3)}$ has no blue $3$-cycle $C_3^{(3)}$; otherwise, the theorem follows by Proposition \ref{C3}.
Since there is no red $K_5^{(3)}$,  the hypergraph $H$ of the blue edges 
has transversal number $\tau(H)\geq 3$.
Then  $H$ is isomorphic to one of the $13$ hypergraphs characterized in Propositions \ref{tauK4} and \ref{taunoK4}.

If  $A\subseteq H$, then every 
red $4$-clique contains two vertices from the blue $4$-clique $K_0\subset A$, and two vertices outside $K_0$; thus any two vertices of $A$ not in $K_0$ form a $K_4^{(3)}$-transversal. 

For $A\not\subseteq H$ the red/blue $K_7^{(3)}$ is not convex if $H\cong C_7^{(3)}$, by Proposition \ref{7cycle},  and  not convex if $H\cong C$, by Proposition \ref{H1}.
If $H$ is among $B^-$, $F$, $F^-$,  $D$ and $D^+$, then
 the convexity requirement in Proposition \ref{V} is violated. \qed
\begin{figure}[http]
\hfill
\begin{center}\begin{tikzpicture}[scale=0.8]
	\setupcoords
	\filldraw[thick,color=col1,fill opacity=0.3] (p01236) -- (p13468) -- (p14578) -- cycle;

	\filldraw[thick,color=col0,fill opacity=0.3] (p01236) -- (p02357) -- (p02478) -- (p04578) -- (p01458) -- (p01456) -- cycle;
	\filldraw[thick,color=col5,fill opacity=0.3] (p23567) -- (p02578) -- (p14578) -- (p01456) -- (p01356) -- (p02356) -- cycle;

	\filldraw[thick,color=col2,fill opacity=0.3] (p02478) -- (p23567) -- (p01236) -- cycle;
	\filldraw[thick,color=col7,fill opacity=0.3] (p23567) -- (p02478) -- (p14578) -- cycle;

	\filldraw[thick,color=col3,fill opacity=0.3] (p23567) -- (p02357) -- (p13468) -- (p01236) -- cycle;
	\filldraw[thick,color=col8,fill opacity=0.3] (p02578) -- (p02478) -- (p14578) -- (p13468) -- cycle;
	
	\filldraw[thick,color=col4,fill opacity=0.3] (p02478) -- (p14578) -- (p13468) -- (p01456) -- cycle;
	\filldraw[thick,color=col6,fill opacity=0.3] (p23567) -- (p01456) -- (p13468) -- (p01236) -- cycle;


	\path[use as bounding box] (current bounding box.south west) rectangle (current bounding box.north east);
	\pgfmathsetmacro{\legy}{7.25}
	\foreach \v in {0,...,8}{
		\filldraw[thick,color=col\v,fill opacity=1] (-9.8,\legy-0.5*\v) rectangle ++(0.4,0.3);
		\node at (-9.1,\legy+0.15-0.5*\v){\v};
	}
\end{tikzpicture}\end{center}
\hfill\mbox{}
\caption{}
\label{9const}
\end{figure}

\section{The nine convex sets construction}
\label{9convex}
We prove the following theorem.
\begin{theorem}
\label{k=5}
There exist $9$ convex sets in the plane such that no point of the
plane is covered more than $5$ times, and there are no two sets among them whose union contains all the $5$-covered points.
\end{theorem}

\begin{proof}
The $9$ convex sets are defined as convex hulls of subsets $12$ points seen in Fig.\ref{9const}. The convex sets are labeled with $\{0,1\ldots,8\}$; the points are labeled with a string of length $5$ specifying the convex sets containing that point:  

	\begin{center}
		$\{0,1,2,3,6\}$,
		$\{0,1,3,5,6\}$,
		$\{0,1,4,5,6\}$,
		$\{0,1,4,5,8\}$,
		$\{0,2,3,5,6\}$,
		$\{0,2,3,5,7\}$,
		
		$\{0,2,4,7,8\}$,
		$\{0,2,5,7,8\}$,
		$\{0,4,5,7,8\}$,
		$\{1,3,4,6,8\}$,
		$\{1,4,5,7,8\}$,
		$\{2,3,5,6,7\}$.\\	
	\end{center}
To help identify the convex hull specification of the $9$ sets
 a color palette is attached indicating the color code of the convex sets.
In the verification we use four auxiliary figures showing separately the following
subfamilies of convex sets:
$NW=\{0,5\}$, $NE=\{1,2,7\}$, $SE=\{4,6\}$, and $SW=\{3,8\}$.\\

\begin{center}
\begin{figure}[http]
\begin{tikzpicture}[scale=0.4]\hskip1cm
	\setupcoords
	\node at (barycentric cs:p01236=1,p01356=1,p02356=1){\huge 0};
	\node at (barycentric cs:p14578=1,p01458=1,p04578=1){\huge 5};
	\filldraw[thick,color=col0,fill opacity=0.3] (p01236) -- (p02357) -- (p02478) -- (p04578) -- (p01458) -- (p01456) -- cycle;
	\filldraw[thick,color=col5,fill opacity=0.3] (p23567) -- (p02578) -- (p14578) -- (p01456) -- (p01356) -- (p02356) -- cycle;
	\lockboundingbox
	\node[anchor=south] at (current bounding box.north west){$NW$};
\end{tikzpicture}
\hfill
\begin{tikzpicture}[scale=0.4]
	\setupcoords
	\node at (barycentric cs:p13468=1,p01456=0.6){\huge 1};
	\node at (barycentric cs:p02356=1.2,p02357=0.5,p23567=1){\huge 2};
	\node at (barycentric cs:p04578=1.2,p02578=0.5,p02478=1){\huge 7};
	\filldraw[thick,color=col1,fill opacity=0.3] (p01236) -- (p13468) -- (p14578) -- cycle;
	\filldraw[thick,color=col2,fill opacity=0.3] (p02478) -- (p23567) -- (p01236) -- cycle;
	\filldraw[thick,color=col7,fill opacity=0.3] (p23567) -- (p02478) -- (p14578) -- cycle;
	\lockboundingbox
	\node[anchor=south] at (current bounding box.north east){$NE$};
\end{tikzpicture}
\hfill\mbox{}
\vspace{2\baselineskip}

\hfill
\begin{tikzpicture}[scale=0.4]
	\setupcoords
	\node at (barycentric cs:p01236=1,p01356=1,p02356=1){\huge 3};
	\node at (barycentric cs:p14578=1,p01458=1,p04578=1){\huge 8};
	\filldraw[thick,color=col3,fill opacity=0.3] (p23567) -- (p02357) -- (p13468) -- (p01236) -- cycle;
	\filldraw[thick,color=col8,fill opacity=0.3] (p02578) -- (p02478) -- (p14578) -- (p13468) -- cycle;
	\lockboundingbox
	\node[anchor=south] at (current bounding box.south west){$SW$};
\end{tikzpicture}
\hfill
\begin{tikzpicture}[scale=0.4]
	\setupcoords
	\node at (barycentric cs:p01236=1,p01356=1,p02356=1){\huge 6};
	\node at (barycentric cs:p14578=1,p01458=1,p04578=1){\huge 4};
	\filldraw[thick,color=col4,fill opacity=0.3] (p02478) -- (p14578) -- (p13468) -- (p01456) -- cycle;
	\filldraw[thick,color=col6,fill opacity=0.3] (p23567) -- (p01456) -- (p13468) -- (p01236) -- cycle;
	\lockboundingbox
	\node[anchor=south] at (current bounding box.south east){$SE$};
\end{tikzpicture}
\hfill\mbox{}
\caption{}
\label{7plus5}
\end{figure}
\end{center}

Let $XX\in$$\{NW,NE,SE,SW\}$.
For  a family $F$ of convex sets we use the notation
$\|XX\|=|XX\cap F|$.
Next we prove that no subfamily of $6$ convex sets from
$F$ have a common point.
Assume, to the contrary that $F$ is a subfamily of $6$ convex sets
with a common point. 

(1) $\|NE\|\leq 2$, since $1\cap 2\cap 7=\varnothing$;

(2) $\|SW\|\leq 1$, because $\|SW\|=2$ would imply $3,8\in F$ and there are only three more sets, $1,4,6$, containing the unique common point of $3$ and $8$;

(3) If $\|NE\|=2$ then $2,7 \in NE\cap F$, since neither the unique point of intersection $01236\in 1\cap 2$ nor $14578\in 1\cap 7$ belongs to six sets. Then it would follow that $\|SE\|=\|SW\|=1$
and  $NW=\{0,5\}\subset F$. However, this is impossible because
$(2\cap 7)\cap 4$ and $(2\cap 7)\cap 6$ are the unique points
$23567$ and $02478$, which are missed by $0$ and $5$, respectively. Thus we conclude that  $\|NE\|\leq 1$. 

Summarizing the observations (1), (2), (3), we obtain that $SE=\{4,6\}\subset F$, $NW=\{0,5\}\subset  F$ and  $\|SW\|=\|NE\|=1$. 
Then the unique point
$01456\in (4\cap 6)\cap(0 \cap 5)$ does not belong to $3\cup 8$, a contradiction. 
Therefore, $|F|<6$, that is no six convex sets have a common point.

Next we verify that there are no two convex sets whose union covers all the $12$ intersection points. It is clear that the union of any two sets belonging to the same subfamily, $NW,NE,SE,$ or $SW$, misses an intersection point. By symmetry, it is enough to verify that the union of any two among the sets $0,1,2,3,4$ does the same, which can be done fast by inspection using the figures.
\end{proof}

\RenewDocumentCommand{\setupcoords}{}{
	\pgfmathsetmacro{\rad}{5}
	\foreach \v [count=\c from 0] in {01456,01356,01236,02356,23567,02357,02578,02478,04578,14578,01458}\coordinate (p\v) at (90+360/11*\c:\rad);
	\coordinate (p13468) at ($(p01356)!1!60:(p01458)$);

	\foreach \v in {01356,02356,02357,02578,04578,01458,01456,01236,02478,14578,23567,13468}{
		\node[vertex] at (p\v){};
		\node[anchor=south] at (p\v){\small\v};
	}
	\colorlet{col0}[rgb]{red>wheel,0,9}
	\colorlet{col1}[rgb]{red>wheel,1,9}
	\colorlet{col2}[rgb]{red>wheel,2,9}
	\colorlet{col3}[rgb]{red>wheel,3,9}
	\colorlet{col4}[rgb]{red>wheel,4,9}
	\colorlet{col5}[rgb]{red>wheel,5,9}
	\colorlet{col6}[rgb]{red>wheel,6,9}
	\colorlet{col7}[rgb]{red>wheel,7,9}
	\colorlet{col8}[rgb]{red>wheel,8,9}
}

\NewDocumentCommand{\drawset}{m}{
	\ifnum #1=0\relax
		\filldraw[thick,color=col0,fill opacity=0.3] (p01236) -- (p02356) -- (p02357) -- (p02578) -- (p02478) -- (p04578) -- (p01458) -- (p01456) -- (p01356) -- cycle;
	\fi
	\ifnum #1=1\relax
		\filldraw[thick,color=col1,fill opacity=0.3] (p01236) -- (p13468) -- (p14578) -- cycle;
	\fi
	\ifnum #1=2\relax
		\filldraw[thick,color=col2,fill opacity=0.3] (p02478) -- (p02578) -- (p02357) -- (p23567) -- (p02356) -- (p01236) -- cycle;
	\fi
	\ifnum #1=3\relax
		\filldraw[thick,color=col3,fill opacity=0.3] (p23567) -- (p02357) -- (p13468) -- (p01236) -- (p02356) -- cycle;
	\fi
	\ifnum #1=4\relax
		\filldraw[thick,color=col4,fill opacity=0.3] (p02478) -- (p04578) -- (p14578) -- (p13468) -- (p01456) -- cycle;
	\fi
	\ifnum #1=5\relax
		\filldraw[thick,color=col5,fill opacity=0.3] (p23567) -- (p02357) -- (p02578) -- (p04578) -- (p14578) -- (p01458) -- (p01456) -- (p01356) -- (p02356) -- cycle;
	\fi
	\ifnum #1=6\relax
		\filldraw[thick,color=col6,fill opacity=0.3] (p23567) -- (p01456) -- (p13468) -- (p01236) -- (p02356) -- cycle;
	\fi
	\ifnum #1=7\relax
		\filldraw[thick,color=col7,fill opacity=0.3] (p23567) -- (p02357) -- (p02578) -- (p02478) -- (p04578) -- (p14578) -- cycle;
	\fi
	\ifnum #1=8\relax
		\filldraw[thick,color=col8,fill opacity=0.3] (p02578) -- (p02478) -- (p04578) -- (p14578) -- (p13468) -- cycle;
	\fi
}
\begin{figure}[h]
\hfill
\begin{tikzpicture}[scale=0.4]
	\setupcoords
	\node at (barycentric cs:p01236=1,p01456=1,p02357=1){\huge 0};
	\drawset{0}
\end{tikzpicture}
\hfill
\begin{tikzpicture}[scale=0.4]
	\setupcoords
	\node at (barycentric cs:p14578=1,p01456=1,p02578=1){\huge 5};
	\drawset{5}
\end{tikzpicture}
\hfill
\begin{tikzpicture}[scale=0.4]
	\setupcoords
	\node at (barycentric cs:p13468=1,p01456=0.6){\huge 1};
	\node at (barycentric cs:p02356=1,p01236=1,p02478=0.5){\huge 2};
	\node at (barycentric cs:p14578=1,p04578=1,p23567=0.5){\huge 7};
	\drawset{1}
	\drawset{2}
	\drawset{7}
\end{tikzpicture}
\hfill\mbox{}
\vspace{2\baselineskip}

\hfill
\begin{tikzpicture}[scale=0.4]
	\setupcoords
	\node at (barycentric cs:p01236=1,p01356=1,p02357=1){\huge 3};
	\node at (barycentric cs:p14578=1,p01458=1,p02578=1){\huge 8};
	\drawset{3}
	\drawset{8}
\end{tikzpicture}
\hfill
\begin{tikzpicture}[scale=0.4]
	\setupcoords
	\node at (barycentric cs:p01236=1,p01456=1,p23567=1){\huge 6};
	\node at (barycentric cs:p14578=1,p01456=1,p02478=1){\huge 4};
	\drawset{4}
	\drawset{6}
\end{tikzpicture}
\hfill\vphantom{.}

\caption{}
\label{11points}
\end{figure}

The construction of the counterexample to prove Theorem \ref{k=5} started with the design of a $2$-representable $f$-vector,  based on Kalai's characterization \cite{K1,K2}. The second  step was to find the appropriate position of the $12$ intersection points of the $5$-tuples taken among the convex hulls generating the $9$ convex sets. 
The $f$-vector of the nerve of our construction is $(f_0,f_1,f_2,f_3,f_4,f_5)=(9,36,61,45,12,0)$, where $f_i$ is equal to the 
number of $i$-dimensional simplices. 
It is worth noting that the realization of this $f$-vector  is not unique. For instance, $11$ of the intersection points can be arranged around a circle as presented in Fig.\ref{11points}.
\section{Extensions} 
\label{arrow}
\subsection{}
Petruska's question prompts several extremal problems  involving convex sets in $\R^d$. For integers $d,t\geq 1$ and $\omega\geq d+1$, let 
$n^*(\omega,t;d)$ be the minimum $n$ such that
there is a $d$-representable red/blue $K_n^{(d+1)}$
with largest red clique size equals to $\omega$ and such that no $t$-vertex  transversal covers all maximum red cliques. The problem of determining  $n^*(\omega,t;d)$ has a substantial difficulty that for $d\geq 2$ no characterization is known for $d$-representable abstract simplicial complexes. 
Theorems \ref{k=4} and \ref{k=5} imply that   $n^*(4,2)=n^*(4,2;2)>7$ and $n^*(5,2)=n^*(5,2;2)\leq 9$.

The analogous extremal problem for $r$-uniform hypergraphs, without the $(r-1)$-representablity requirement, was proposed by Erd\H{os} \cite{E}.
Let $n(\omega,t;r)$ be  the minimum $n$ satisfying the property that
there is an $r$-uniform hypergraph
with largest clique size equal to $\omega$ and such that no $t$-vertex  transversal covers all maximum cliques.
It is clear that $n(\omega,t;r)\leq n^*(\omega,t;r-1)$. Another less obvious relationship
between the functions $n$ and $n^*$ is as follows.
\begin{lemma}
\label{interconnect} For $t,d\geq 2$,
 $\; n^*(\omega,t;d-1)\geq n^*(\omega,t-1;d-1)+1\geq n(\omega,t-1;d)+1$.
\end{lemma}
\begin{proof}   Let
$H$ be a $d$-representable $(d-1)$-uniform `witness hypergraph' of order $n^*(\omega,t;d-1)$ such that $\omega(H)=\omega$, 
and there is no $t$-vertex transversal for its $\omega$-cliques.
Removing a vertex $v_0$ from $H$ together with all edges of $H$ containing $v_0$,
one obtains a $d$-representable $(d-1)$-uniform witness hypergraph $H^-$  such that no $t-1$ vertices cover  its $\omega$-cliques. Thus 
$ n^*(\omega,t-1;d-1)\leq n^*(\omega,t;d-1)-1$ follows.
The second inequality is obvious by definition.  
\end{proof}
Several results and conjectures were established for $n(\omega,1;3)$ in \cite{GYLT,SzP,Tu}. For example, $n(5,1;3)=8$. Applying Lemma \ref{interconnect} we obtain that
 $n^*(5,2;2)\geq 9$. 
 This bound combined with Theorem \ref{k=4} results in the value $n^*(5,2;2)=9$, which we state as a sharpening of Theorem \ref{k=4}.
   \begin{theorem} 
\label{k=5sharp} 
If a convex red/blue $K_8^{(3)}$ 
contains no red   $K_6^{(3)}$,  then there exist two vertices
such that every red $K_5^{(3)}$ contains at least one of them.
Moreover, the claim is not true if $K_8^{(3)}$ is replaced with $K_9^{(3)}$.\qed
\end{theorem}

\subsection{}
\label{t=1}
We know very little about 
the functions $n(\omega,t;r)$ and $n^*(\omega,t;r-1)$, even for $r=3$ and $t=1$.
To simplify the notation here, set  $n(\omega)=n(\omega,1;3)$ and  
$n^*(\omega)= n^*(\omega,1;2)$. 
Table (\ref{tab:1}) shows  $n(\omega)$\footnote{obtained by A. Jobson, A. K\'ezdy, J. Lehel, and T. Pervenecki, 
The intersection of the maximum cliques
in $3$-uniform hypergraphs (in preparation)}, for $\omega\leq 12$;  
we shall verify that $n^*(\omega)=n(\omega)$ in that range,  except the case $\omega=11$.
\begin{center}
\begin{table}[htp]
\sidecaption
\caption{}
\label{tab:1}   
\hskip1.74cm\begin{tabular}{|c||c|c|c|c|c|c|c|c|c|c|c}\hline
$\omega$&3&4&5&6&7&8&9&10&11&12&\\ \hline
$n$&5&6&8&9&10&12&13&14&15&17& \\ \hline
$n^*$&5&6&8&9&10&12&13&14&{\bf 16}&17& \\ \hline
\end{tabular}
\end{table}
\end{center}
According to a conjecture of Szemer\'edi and Petruska \cite{SzP}, 
 $n(\omega)= \omega +m$ for $\omega =\binom{m+1}{2}+1$. Furthermore, for every $m\geq 4$, the  extremal system of $\omega$-cliques is unique; denote the corresponding $3$-uniform hypergraph formed by the triples lying in the cliques by $WH(m)$. These witness hypergraphs
all  contain the partial hypergraph $WH(3)$ of order $7+3=10$ defined  in terms of a red/blue $K_{10}^{(3)}$ as follows. 
Label the vertices 
with 
$1,2,3,4,5,6$,$a,b,c,d$; let the vertices with alphabetical label form a blue $K_4^{(3)}$, and let $\{1,a,b\}$,$\{2,a,c\}$,$\{3,a,d\}$,$\{4,b,c\}$,$\{5,b,d\}$,$\{6,c,d\}$ and no other triples  of $K_{10}^{(3)}$ be colored blue. 

\begin{proposition}
\label{4hole}   The red/blue $K_{10}^{(3)}$ corresponding to $WH(3)$  is not convex.
\end{proposition}
\begin{proof} Assume to the contrary that there are ten convex sets labeled with the vertex labels and the red subhypergraph of the red/blue $10$-clique  is the $2$-skeleton of their nerve. By Helly's theorem, there are six points
\begin{center}
\(
\begin{array}{cccccc}
P_{ab}&\in& a\cap b\cap 2\cap3 \cap 4\cap 5\cap  6,\quad
&P_{ac}&\in & a\cap c\cap 1\cap 3\cap 4\cap 5\cap  6,\\
P_{ad}&\in & a\cap d\cap 1\cap2 \cap 4\cap 5\cap  6,\quad
&P_{bc}&\in & b\cap c\cap 1\cap 2\cap 3\cap 5\cap  6,\\
P_{bd}&\in & b\cap d\cap 1\cap2 \cap 3\cap 4\cap  6,\quad
&P_{cd}&\in & c\cap d\cap 1\cap 2\cap 3\cap 4\cap  5.
\end{array}
\)
\end{center}
If these six points are not vertices of a convex hexagon, then 
there is one point, say $P_{ab}$, in the convex hull of  the other points. Since these points all  belong to set $1$, we have $P_{ab}\in a\cap b\cap 1$; this is not possible, since $\{1,a,b\}$ is a blue edge. 

Assume now that those six points are vertices of a convex hexagon. 
If there is a point 
 $Q\in(\overline{P_{ab}P_{bc}}\cup\overline{P_{ab}P_{bd}}\cup\overline{P_{bc}P_{bd}})\cap
\overline{P_{ac}P_{ad}}$, then we have $Q\in a\cap b\cap 1$, contradicting that $\{1,a,b\}$ is a blue edge. 
W.l.o.g. assume that $P_{ac}$ and $P_{ad}$ are not consecutive vertices of the hexagon.
The observation above shows that the points $P_{ab}$, $P_{bc}$, and $P_{bd}$ are on the same side of the line
$\overleftrightarrow{P_{ac}P_{ad}}$, thus $P_{cd}$  is on the opposite side of this line. 
Then there is a point $Q\in \overline{P_{ac}P_{ad}}\cap \overline{P_{bc}P_{cd}}$ implying that $Q\in a\cap c\cap 2$, contradicting that $\{2,a,c\}$ is blue.
\end{proof}

We proved\footnote{Ibid.}  that the extremal system $WH(4)$ of the $11$-cliques is unique.  
Because the $3$-uniform witness hypergraph corresponding to $WH(3)$ is not $2$-representable,  and $WH(3)\subset WH(4)$, we obtain $n^*(11)>n(11) =15$.
The value $n^*(11)=16$ 
is justified by the following polygon construction. This construction also yields a few more values of $n^*$.\vspace{\baselineskip}

\noindent {\it Polygon construction.} For $k\geq 3$, let $\mathcal{R}_{k}$ be the regular $k$-gon;  
take the convex hull of every set of $(k-1)$ vertices, and take the convex hull of 
every set of $\lceil k/2\rceil$ consecutive vertices. Thus we obtain $n^*=2k$ convex sets
such that, with the exception of the vertices of $\mathcal{R}_{k}$, the points of $\R^2$ are covered less than $\omega=k-1+\lceil k/2\rceil$ times. Most importantly, we obtain a few values of  $n^*$ in Table \ref{tab:2}.

\begin{center}
\begin{table}[htp]
\sidecaption
\caption{}
\label{tab:2} 
\hskip1.74cm\begin{tabular}{|c||c|c|c|c|c|c|c|}\hline
$k$&3&4&5&6&7&8\\ \hline
$\omega$&4&5&7&8&10&11\\ \hline
$n^*$&6&8&10&12&14&16 \\ \hline
\end{tabular}
\end{table}
\end{center}
For the missing values $n^*(9)$ and  $n^*(12)$ we extend the Polygon construction for $k=5$ and $k=7$, respectively, by repeating two $(k-1)$-gons missing consecutive vertices $P,Q$ of  $\mathcal{R}_{k}$, and including the segment $\overline{PQ}$.
The three new convex sets 
increase the point cover by two, thus yielding the  values $n^*(9)=n^*(7+2)=n^*(7)+3=13$
and $n^*(12)=n^*(10+2)=n^*(10)+3=17$. To obtain constructions verifying that  $n^*(3)=5$ and $n^*(6)=9$ 
we use the following triangle construction.\vspace{\baselineskip}

\noindent{\it Triangle construction.} Let $P,Q,R\in\R^2$ be 
non-collinear points; define the family $\mathcal{F}(\omega)$ of $\omega +\lceil\omega/2\rceil$ segments: $ \lceil\omega/2\rceil$ copies of $\overline{PR}$, and
$ \lfloor\omega/2\rfloor$ copies of each segment, $\overline{PQ}$ and  $\overline{RQ}$; and for $\omega$ odd, we include the single point $Q$ to the family.\vspace{\baselineskip}

It is worth noting that for $\omega=7$  (the case $m=3$ in the Szemer\'edi and Petruska conjecture), there are two extremal systems of $7$-cliques; the witness hypergraph corresponding to $WH(3)$ is not $2$-representable, but the second one coincides with the $3$-uniform intersection hypergraph of the Polygon construction for $k=5$.

\subsection{}
\label{further}
Turning back to Petruska's original question we show that Theorem \ref{k=4}  is sharp in the sense that $K_7^{(3)}$ cannot be replaced  by $K_8^{(3)}$. This claim    follows from the more general proposition that $n^*(k,2)\leq 2k$.
\begin{proposition}
\label{2kbound}
 There are $2k$ convex sets
 in $\R^2$ satisfying that each point of $\R^2$ is covered at most $k$ times, and the $k$-covered points are not contained in the union of two among the convex sets.
\end{proposition}
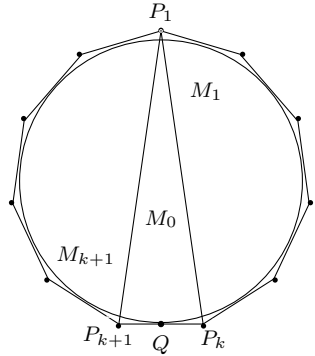
\begin{figure}[htp]
\centering
\begin{tikzpicture}
\tikzstyle{txt}  = [circle, minimum width=1pt, draw=white, inner sep=0pt]
\tikzstyle{A} = [circle, draw=black!,fill=black,minimum width=2pt, inner sep=.5pt]
\node[draw=none,minimum size=4cm,regular polygon,regular polygon sides=11] (a) {};

\foreach \x in {1,2,...,6,7,8,9,10,11}
  \fill (a.corner \x) circle[radius=1pt];
  \draw[rotate=0] (0,0) ellipse (1.88 and 1.88);
  \draw (.56,-1.9)--(-.56,-1.9)--
  (0,2)--(.56,-1.9)--(1.5,-1.3)--(1.97,-.3)--(1.82,.8)--(1.05,1.69) --(0,2);
    \draw (-.56,-1.9)--(-1.5,-1.3)--(-1.97,-.3)--(-1.82,.8)--(-1.05,1.69)--(0,2)   ;
  \node[A,label=below:{\small $Q$}] (Q)at (0,-1.9){};
   \node[txt,label=above:{\small $P_1$}] (Q)at (0,2){};
    \node[txt,label=below:{\small $P_k$}] (pk)at (0.7,-1.85){};
     \node[txt,label=below:{\small $P_{k+1}$}] (pk1)at (-0.7,-1.8){};
     \node()at(.6,1.2) {\small $M_1$};   \node()at(-1,-1){\small$M_{k+1}$};   
     \node()at(0,-0.5) {\small $M_0$};
\end{tikzpicture}
\end{figure}

\begin{proof}
Let $\mathcal{R}_{2k-1}$ be the regular $(2k-1)$-gon with vertices
 $P_1,\ldots,P_{2k-1}$. Let $M_0$ be the disk of the circle inscribed into $\mathcal{R}_{2k-1}$;  for $i=1,\ldots,2k-1$, let $M_i$ be the $k$-gon cell defined by the convex hull of the $k$ consecutive vertices of $\mathcal{R}_{2k-1}$ starting at $P_i$. 
 
 (1) Observe that each set  $M_i$,  $i\neq 1$, contains either $P_1$ or $P_k$ and not both. 
Therefore, by symmetry, every vertex of $\mathcal{R}_{2k-1}$ is $k$-covered.
Moreover, the non-empty intersection of $k$ cells is a single vertex of
$\mathcal{R}_{2k-1}$. This also  implies  
that no point of $\R^2$ is covered more than $k$-times. 

(2) The disk $M_0$ does not contain any vertex of $\mathcal{R}_{2k-1}$, furthermore, we need at least two $k$-gon cells to cover all vertices.
Without loss of generality assume that $P_1,\dots, P_{2k-1}\in M_1\cup M_{k+1}$. Let $Q$ be the midpoint of the segment $\overline{P_{k+1}P_k}$. Since $Q\in M_0\cap\left(\bigcap\limits_{i=2}^{k}M_i\right)$  is a $k$-covered point, furthermore, $Q\notin M_1\cup M_{k+1}$, a third convex set is necessary to include all $k$-covered points.
\end{proof}
Proposition \ref{2kbound} applied with $k=4$ proves the stronger version of 
Theorem \ref{k=4} as follows.
\begin{theorem} 
\label{k=4sharp} 
If a convex red/blue clique $K_7^{(3)}$ 
contains no red   $K_5^{(3)}$,  then there exist two vertices
such that every red $K_4^{(3)}$ contains at least one of them.
Moreover, the claim is not true if $K_7^{(3)}$ is replaced with $K_8^{(3)}$.\qed
\end{theorem}

\section{Concluding remarks}
\label{remarks}
 
 \subsection{}
 
In terms of the function $n^*(\omega,2)$ defined in Section \ref{arrow}, Petruska's question becomes an extremal problem described by the claim that  
in every family of $n< n^*(k,2)$  planar compact convex sets 
such that no point of the plane is covered  $(k+1)$-times, there are two members 
whose union contains all $k$-covered points.

  Lemma \ref{interconnect} applied with $t=d=2$ relates $n^*(\omega,2)=n^*(\omega,2;2)$ to the corresponding extremal function $n(\omega)=n(\omega,1;3)$ (which is defined for $3$-uniform hypergraphs without the $2$-representability requirement), and leads to  the bound $n^*(\omega,2)\geq n(\omega)+1$.
This inequality combined with the bound  $n(k)\leq (n(k)-k)^2+2(n(k)-k)$ 
due to Tuza \cite{Tu} imply a lower bound,
and Proposition \ref{2kbound} yields an upper bound on  $n^*(k,2)$ as follows: 
$$  k+\sqrt{k}\leq n^*(k,2)\leq 2k.$$ It is worth noting that 
the sharp bound  $n(k)\leq {n(k)-k+2\choose 2}$  conjectured by Szemer\'edi and Petruska \cite{SzP} would yield only the slight improvement $n^*(k,2)\geq k+\sqrt{2k}-O(1)$. 
We are wondering whether $n^*(k,2)\geq (1+\epsilon)k -O(1)$  is true with some $\epsilon>0$.
\subsection{}
The red/blue clique model introduced here allows the discussion of the intersection and non-intersection patterns of convex sets simultaneously, and  in terms of forbidden red/blue subconfigurations. A few general properties of convex red/blue cliques   are included in Section \ref{forbidden}; although, it is not obvious how much convexity must be used in proving covering theorems like Theorem \ref{k=4}. Actually, by conducting  a computer search on $3$-uniform hypergraphs of order $7$, we did verify a
more general version of Theorem \ref{k=4}; the claim  remains true when the convexity requirement is replaced with a bit lighter condition  imposed on the red/blue clique, namely the $2$-representability of the $f$-vector of its red subhypergraph.

An inventory of the simplices in an abstract simplicial complex is usually expressed by the $f$-vector {\bf {f}}${}=(f_0,f_1,f_2,\cdots)$, where $f_k$ is the number of simplices with $k+1$ vertices. Let the {\it $f$-vector of a red/blue clique} be defined as the $f$-vector of the red abstract simplicial complex of the cliques of all sizes generated by the red subhypergraph. 
 For instance, the $f$-vector of a blue Fano-plane is $(7,21,28,7)$.
An $f$-vector {\bf {f}}${}=(f_0,f_1,f_2,\ldots)$ is {\it $d$-representable} if there is a family of convex sets in $\R^d$ such that
the $f$-vector of the nerve of that family is 
equal to  $(f_0,f_1,\ldots)$. 

\begin{theorem} 
\label{vectoros}
Let {\bf {f}}${}=(f_0,f_1,f_2,f_3,0)$ be the $f$-vector of a red/blue $7$-clique. If 
{\bf {f}} is $2$-representable, then the red $4$-cliques have a $2$-vertex transversal.
\end{theorem}
Kalai's theorem (\cite{K1,K2}) establishes  a sufficient and necessary numerical condition for an $f$-vector to be 
$d$-representable. This makes possible the computer verification of Theorem \ref{vectoros}, which in turn, implies Theorem \ref{k=4} since, by definition, if a red/blue clique is convex in $\R^d$, 
then its  $f$-vector is $d$-representable, as well.  (The converse is not true, for instance, any tree as a $1$-complex  has $f$-vector $(n,n-1)$, but not all trees are  interval graphs, for $n\geq 7$.)
 
 All the forbidden red/blue cliques described in Section \ref{forbidden}
remain obstructions against the $2$-representability of their $f$-vectors.
For instance, $(7,21,28,14,0)$, the $f$-vector of the blue chordless $7$-cycle is not $2$-representable (yielding another verification that $C_7^{(3)}$ is not convex).
In spite of this, the proof of Theorem \ref{k=4}  given here does not rise to the level of a 
 combinatorial proof of Theorem \ref{vectoros}, mainly because  the family of complexes with $2$-representable $f$-vector is not closed under taking suhypergraphs, which is obvious in case of $2$-representable complexes.\\

 \noindent {\bf Acknowledgment.} We would like to thank Imre B\'ar\'any
 for calling our attention to the related  results of Berge and Ghouila-Houri on the intersections of convex sets.

\end{document}